\documentclass{article}
\usepackage[utf8]{inputenc}
\usepackage{mycommands}
\usepackage{geometry}
\usepackage[numbers]{natbib}
\usepackage{enumitem}
\usepackage{setspace}
\title{Estimation of subsidiary performance metrics under optimal policies}
\author{%
  Zhaoqi Li\footnote{Department of Statistics, University of Washington, \texttt{zli9@uw.edu}}%
  \and Houssam Nassif\footnote{Meta, \texttt{houssam.nassif@gmail.com}}%
  \and Alex Luedtke\footnote{Department of Statistics, University of Washington,  \texttt{aluedtke@uw.edu}}
}
\date{\today}

\begin{document}
\allowdisplaybreaks 

\maketitle

\begin{abstract}
    In policy learning, the goal is typically to optimize a primary performance metric, but other subsidiary metrics often also warrant attention. This paper presents two strategies for evaluating these subsidiary metrics under a policy that is optimal for the primary one. The first relies on a novel margin condition that facilitates Wald-type inference. Under this and other regularity conditions, we show that the one-step corrected estimator is efficient. Despite the utility of this margin condition, it places strong restrictions on how the subsidiary metric behaves for nearly optimal policies, which may not hold in practice. We therefore introduce alternative, two-stage strategies that do not require a margin condition. The first stage constructs a set of candidate policies and the second builds a uniform confidence interval over this set. We provide numerical simulations to evaluate the performance of these methods in different scenarios. 
\end{abstract}

\section{Introduction}

\subsection{Literature Review}

Many fields are interested in learning policies that map from individual-level characteristics to a choice of action.  
The policies that result in the best possible mean of a subsequent outcome are often referred to as optimal policies \cite{athey2021policy}. 
For example, in biomedical sciences the action may take the form of a treatment allocation and the outcome may be disease remission \cite{ling2021heterogeneous}, whereas in digital marketing the action and outcome may be a recommendation and click-through rate, respectively \cite{hill2017efficient}. 
There have been a variety of methods developed for estimating optimal policies. These methods include regression-based estimators such as Q-learning \cite{qian2011performance}, outcome-weighted learning \cite{zhao2012estimating}, and doubly robust approaches \citep{dudik2011doubly,zhang2013robust}, among others. 
Performance guarantees for these methods have been established by several authors \cite{athey2021policy,qian2011performance,zhao2012estimating,luedtke2020performance}. 

An estimated policy is unlikely to be implemented unless confidence intervals characterizing its performance are available \cite{shi2020statistical}. 
These performance metrics may take the form of the remission rate of all patients or the click-through rate of all customers. 
In both of these examples, the metric is the value of the optimal policy in the population, better known as the optimal value. 
Inference about the optimal value is well-studied when there is only one outcome of interest  \cite{luedtke2020performance,liu2021efficient}. 
Several works have shown that one-step estimators and targeted minimum loss based estimators are efficient under conditions \cite{van2015targeted,chambaz2017targeted}. In particular, these works require a non-exceptional law condition that states that the conditional average action effect does not concentrate mass at zero \cite{robins2004optimal}. 
Alternative strategies have been developed for constructing confidence intervals for the optimal value even when this condition fails \cite{Alex16,chakraborty2013inference}. 

Though most existing methodological works on policy learning focus on optimizing for a single performance metric, in real-world settings there are often multiple other subsidiary performance metrics that are also of interest \cite{boominathan2020treatment,bica2021real}. 
These metrics may correspond to different summaries of the outcome, such as the median, rather than the mean, and time to disease remission \cite{phillips2020statistical}. 
Alternatively, they may summarize several different outcomes rather than just a single one. 
For example, when learning a treatment allocation, symptom reduction may be considered alongside prognosis \cite{freemantle2003composite}. 
Most existing approaches for incorporating multiple outcomes involve combining them into a composite outcome and then using policy learning methods designed for single-outcome settings \cite{butler2018incorporating}.
In settings where the actions recommended by experts are recorded in the dataset, \citeauthor{murray2016utility} provided a means to construct a composite outcome in an automated fashion \cite{murray2016utility}. 
However, when expert recommendations are not available, composite-outcome-based approaches require investigators to construct the composite outcome in some other way, which often ends up being somewhat arbitrary \cite{luckett2021estimation}. 
Some alternative approaches do not require the construction of a composite outcome. One such approach involves learning a policy that returns a set of recommended actions, rather than a single one \cite{laber2014set}. Each of the actions in this set should yield a desirable result for at least some of the outcomes. For settings where there is a primary outcome of interest, another approach involves using other secondary outcomes to define constraints that any selected policy must satisfy \cite{linn2015chapter,wang2018learning}. 

In cases where a single outcome is of primary interest and others are only of secondary interest, a preferred approach may be to optimize only for this one outcome, while still making inferences about the effect of the policy on the subsidiary outcomes. 
For example, just as the side effects of any new medical intervention must be assessed along with its effect on the primary outcome of interest \citep{FDA2006}, the side effects of a new treatment policy should be assessed as well \citep{linn2015chapter}. 
As another example, if a company optimizes a policy for customer acquisition, it must also consider the impact the policy will have on customer retention \cite{afeche2017customer}. 
In this work, we provide a systematic approach for assessing the impact of a policy that is optimized for some primary outcome on other, subsidiary outcomes.

\subsection{Notation and objectives}
Let $X\in\X$ be a feature, $A\in\{0,1\}$ a binary action, and $Y\in\Y$ an outcome that is observed after the action. 
This outcome may be multivariate. Let $\M$ be a nonparametric model consisting of possible joint distributions $P$ of $(X,A,Y)$. Our sample consists of $n$ independent and identically distributed draws $(X_i,A_i,Y_i)_{i=1}^n$ from $P_0\in \M$. 
Let $\Pi$ be a set of policies $\X\to\{0,1\}$ that take as input a feature and take action 0 or 1. 
For a given policy $\pi\in\Pi$, let $\Omega_{\pi}(P)$ be a real-valued primary performance metric for the policy $\pi$ under sampling from $P$, where we assume that larger values of this metric are considered preferable. Further let $\Psi_{\pi}(P)$ be a real-valued subsidiary performance metric for $\pi$ under $P$. 
For example, when $Y$ is a primary-subsidiary outcome pair $(Y^\star,Y^\dagger)\in\mathbb{R}^2$, these metrics could be the covariate-adjusted means of these two outcomes \cite{luedtke2020performance,Alex16}, that is, $\Omega_{\pi}(P)=\int \E_P[Y^\star|A=\pi(x),X=x] dP(x)$, and $\Psi_{\pi}(P)=\int \E_P[Y^\dagger|A=\pi(x),X=x] dP(x)$. Alternatively, the outcome $Y$ may be univariate and the primary performance metric may be equal to the mean $\Omega_{\pi}(P)=\int \E_P[Y|A=\pi(x),X=x] dP(x)$ while the subsidiary metric may be equal to the covariate-adjusted probability that the outcome exceeds a specified value $t$, namely $ \Psi_{\pi}(P)=\int P\{Y> t|A=\pi(x),X=x\} dP(x)$. 
We refer to $\Omega_\pi(P_0)$ and $\Psi_\pi(P_0)$ as the $\Omega$-performance and $\Psi$-performance of the policy $\pi$.

For $P\in\mathcal{P}$, let $\Pi_P^*$ denote the set of optimal policy with respect to the primary performance metric, that is, $\Pi_P^*:=\{\pi\in\Pi:\Omega_\pi(P)=\sup_{\pi'\in\Pi}\Omega_{\pi'}(P)\}$. We denote a generic element of this set by $\pi_P^*$. We refer to elements of $\Pi^*:=\Pi_{P_0}^*$ as $\Omega$-optimal policies and denote a generic element by $\pi^*$. 
We assume throughout that $\Pi^*$ is nonempty, and note that in general this set may contain more than one policy. 
We are interested in making inferences about the subsidiary performance metric $\Psi_{\pi}(P_0)$ for $\Omega$-optimal policies. 
Letting $\psi_0^\ell=\inf_{\pi\in\Pi^*}\Psi_\pi(P_0)$ and $\psi_0^u=\sup_{\pi\in\Pi^*}\Psi_\pi(P_0)$, our objective is to construct a confidence interval for the range of possible $\Psi$-performances under an $\Omega$-optimal policy, that is, develop a confidence interval that is a superset of $[\psi_0^\ell,\psi_0^u]$ with a specified asymptotic probability.

When there is only one $\Omega$-optimal policy, our objective is to determine the $\Psi$-performance of this policy, denoted by $\psi_0:=\psi_0^\ell=\psi_0^u$. When there are multiple $\Omega$-optimal policies, $\psi_0^\ell$ may be less than $\psi_0^u$, and the upper and lower bounds of our interval inform on the most extreme $\Psi$-performances that can be attained from an $\Omega$-optimal policy. For example, if larger values of $\Psi$ are preferable, then the upper confidence bound on $\psi_0^u$ informs about the best achievable $\Psi$-performance by an $\Omega$-optimal policy. 
Such a policy can be shown to be one of several policies that fall along the Pareto front of the two-objective optimization problem that seeks to maximize $\Omega$ and $\Psi$. The Pareto front denotes the set of policies for which there is not a policy that performs better with respect to one of the two metrics and no worse with respect to the other. 
The difference between inferring about $\psi_0^u$ and multi-objective optimization is that the policy with the best $\Psi$ performance is primarily optimized with respect to one performance metric $\Omega$, while multi-objective optimization optimizes several performance metrics simultaneously \cite{gunantara2018review,deb2014multi,bentley1998finding}.

We present our proposed confidence intervals in the next two sections. 
When doing so, we consider two separate cases. 
In Section~\ref{sec:inf_under_margin_cond}, we begin with an easier and more specialized case, where the performance metrics $\Omega_\pi$ and $\Psi_\pi$ are assumed to be the covariate-adjusted means of a primary outcome ($Y^\star$) and subsidiary outcome ($Y^\dagger$), there is assumed to be a unique $\Omega$-optimal policy $\pi^*$ over an unrestricted policy class $\Pi$, and a certain margin condition holds. 
In Section~\ref{sec:general_inference}, we move on to a harder and more general case, where $\Omega_\pi$ and $\Psi_\pi$ are arbitrary smooth parameters and there may be multiple $\Omega$-optimal policies. 

\section{Wald-type inference under a margin assumption}\label{sec:inf_under_margin_cond}

In this section, we focus on the case where $\Omega_{\pi}(P)=\int \E_P[Y^\star|A=\pi(x),X=x] dP(x)$ and $\Psi_{\pi}(P)=\int \E_P[Y^\dagger|A=\pi(x),X=x] dP(x)$ for a primary-subsidiary outcome pair $(Y^\star,Y^\dagger)$ and, moreover, the policy class $\Pi$ is unrestricted. 
We aim to build on existing works that evaluate the $\Omega$-performance of an $\Omega$-optimal policy \citep{van2015targeted,Alex16}. These works have shown that a simple estimation strategy is efficient under a non-exceptional law condition that makes the $\Omega$-optimal rule unique \citep{robins2004optimal}. In this case, $\psi_0^\ell=\psi_0^u$ and we write $\psi_0=\psi_0^\ell=\psi_0^u$. 
This strategy first obtains an estimate $\widehat{\pi}$ of the $\Omega$-optimal rule, and then constructs a standard one-step estimator of $\Omega_{\widehat{\pi}}(P_0)$.
Heuristically speaking, pursuing estimation of $\Omega_{\widehat{\pi}}(P_0)$, rather than $\Omega_{\pi^*}(P_0)$, introduces only negligible bias because $\widehat{\pi}$ should be a near-maximizer of $\Omega_\pi(P_0)$. Hence, similarly to the fact that $f(x)-f(x^*)=o(|x^*-x|)$ for a differentiable function $f : \mathbb{R}\rightarrow\mathbb{R}$ with maximizer $x^*$, the error induced by replacing $\pi^*$ by $\widehat{\pi}$ in the functional $\pi\mapsto \Omega_{\pi}(P_0)$ should be second-order. 
In this section, we study the extent to which a standard one-step estimator of $\Psi_{\widehat{\pi}}(P_0)$ will yield an asymptotically normal and efficient estimator of $\Psi(P_0)$. 
This study is important since, if the standard one-step estimator satisfies these properties under only mild conditions, then there is little reason to develop alternative methods.

We now discuss a key condition that we will require to establish the efficiency of a standard one-step estimator for $\Psi(P_0)$, along with the validity of corresponding Wald-type confidence intervals. 
Define the function $q_b(P)(x):=\E_P[Y^\star|A=1,X=x]-\E_P[Y^\star|A=0,X=x]$ to be the conditional average treatment effect on the primary outcome and $s_b(P)(x):=\E_P[Y^\dag|A=1,X=x]-\E_P[Y^\dag|A=0,X=x]$ to be the conditional average treatment effect on the subsidiary outcome. We refer to these functions as the primary CATE and subsidiary CATE, respectively. 
We use the shorthand notation $q_{b,0}:=q_b(P_0)$ and $s_{b,0}:=s_b(P_0)$. 

\begin{condition}[Margin condition between $Y^\dag$ and $Y^\star$]\label{cond:margin_ZY}
For some $C_1>0$ and $\zeta>2$,
\begin{align}
P_0\left(\left|s_{b,0}(X)\right|\geq C_1t\left|q_{b,0}(X)\right|\right)\leq t^{-\zeta},\qquad \ \textnormal{ for all }t>1. \label{eq:margin_ZYDisplay}
\end{align}
\end{condition}
When this condition holds, $\left|q_{b,0}(X)\right|\not=0$ with $P_0$-probability one.
Hence, this condition is a strengthening of the usual non-exceptional law condition \citep{robins2004optimal} that is required when the $\Psi$ and $\Omega$ performance metrics coincide. 
To ensure the validity of the standard one-step estimator, some form of strengthening appears to be needed to make up for the fact that $\pi^*$ is defined as a maximizer in $\pi$ of $\Omega_\pi(P_0)$, rather than $\Psi_\pi(P_0)$. 
Indeed, the estimation error of this estimator $\hat{\psi}_{\hat{\pi}}$ can be decomposed as 
\begin{align*}
    \widehat{\psi}_{\hat{\pi}} - \Psi_{\pi^*}(P_0)&= \left[\widehat{\psi}_{\hat{\pi}} - \Psi_{\widehat{\pi}}(P_0) \right] + \left[ \Psi_{\widehat{\pi}}(P_0) - \Psi_{\pi^*}(P_0)\right].
\end{align*}
The fact that $\hat{\psi}_{\hat{\pi}}$ is a one-step estimator of $\Psi_{\widehat{\pi}}(P_0)$ should imply that the first term will be small. However, since $\pi^*$ is not necessarily an optimizer for $\Psi$, it is possible that $\Omega_{\hat{\pi}}(P_0)$ is close to $\Omega_{\pi^*}(P_0)$ while $\Psi_{\hat{\pi}}(P_0)$ is far from $\Psi_{\pi^*}(P_0)$ --- see Figure \ref{fig:phi_vs_psi} for an illustration of this possibility. 
\begin{figure}[!htb]
    \centering
    \includegraphics[scale=0.4]{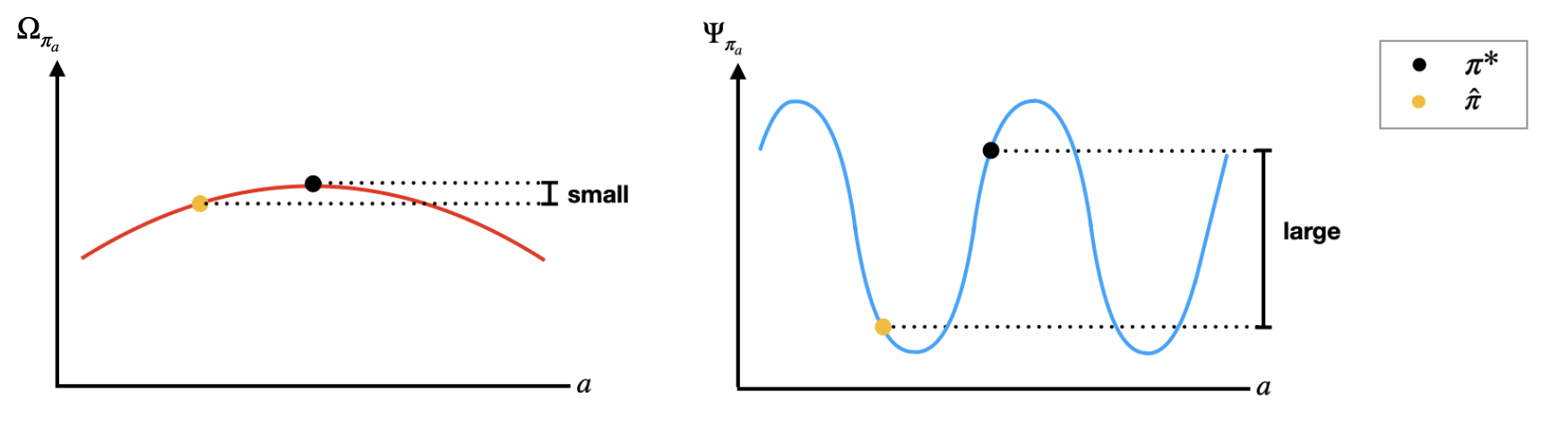}
    \caption{Plot of primary and subsidiary performance metrics for an estimated policy given the threshold policy class $\Pi=\{\mathbf{1}(x\geq a): a\in\mathbb{R}\}$. The estimator $\hat{\pi}$ performs well in the sense that the $\Omega$-regret $\Omega_{\pi^*}(P_0)-\Omega_{\hat{\pi}}(P_0)$ is small, which is to be expected since $\pi^*$ is defined to be an $\Omega$-optimal rule. Nevertheless, in principle the $\Psi$-regret $\Psi_{\pi^*}(P_0)-\Psi_{\hat{\pi}}(P_0)$ could still be large, since the $\Psi$-value function $\pi\mapsto\Psi_\pi(P_0)$ may be markedly different from the $\Omega$-value function. Though a similar phenomenon can occur for unrestricted policy classes, which are our focus in this section, the infinite-dimensional nature of these classes precludes their visualization.}
    \label{fig:phi_vs_psi}
\end{figure}
Therefore, we need a condition to characterize the flatness of the $\Psi$ performance surface relative to that of $\Omega$. 
This flatness can be characterized by studying the absolute CATE ratio $\left|q_{b,0}(X)\right|/\left|s_{b,0}(X)\right|$, where we use the convention that $b/0=+\infty$ for $b>0$ and we recall that $\left|q_{b,0}(X)\right|=0$ with probability zero under \eqref{eq:margin_ZYDisplay}. 
Condition~\ref{cond:margin_ZY} imposes that the absolute CATE ratio can only concentrate vanishingly little mass near zero when $X\sim P_0$. 
This certainly holds in the extreme case where, within each level $x$ of the covariates, the magnitude of the expected effect of the action on the primary outcome, namely $|q_{b,0}(x)|$, is at least as large as the magnitude of its effect on the subsidiary outcome, namely $|s_{b,0}(x)|$. 
It also allows for scenarios where the magnitude $|s_{b,0}(x)|$ is much larger than $|q_{b,0}(x)|$ for certain features $x$ with sufficiently small probability of occurrence. 
However, it can fail to hold when there are some feature levels where the action has no effect on the primary outcome and yet does have one on the subsidiary outcomes; this can occur, for example, if the primary outcome is cancer remission and the subsidiary outcome captures side effects induced by chemotherapy. 
Though Condition~\ref{cond:margin_ZY} may be strong, we were unable to show the validity of the standard one-step estimator without it. Therefore, in the remainder of this section we assume that this condition holds, and we refer the reader to the next section for a method that is valid even when it does not.

In the special case where $Y^\dag=Y^\star$ a.s., the asymptotic normality and efficiency of the one-step estimator have previously been justified by establishing the 
pathwise differentiability of the $\Omega$-performance of an $\Omega$-optimal policy  \cite{van2015targeted}. 
We follow a similar approach here when considering cases where $Y^\dag$ and $Y^\star$ may differ. In particular, we establish the pathwise differentiability of $\Psi^* : P\mapsto \sup_{\pi\in\Pi_P^*}\Psi_\pi(P)$ in what follows. 
When doing this, we will need to impose Condition~\ref{cond:margin_ZY}, along with an additional margin condition that is inspired by ones previously assumed in the policy learning \citep{qian2011performance,Alex16} and classification \citep{audibert2007fast} literatures. 
\begin{condition}[Margin condition for $Y^\star$]\label{cond:margin_Y}
For some $\gamma>\frac{1}{\zeta}$, 
\begin{align}
    P_0\left(0< \left|q_{b, 0}(X)\right| \leq t\right) \lesssim t^{\gamma}\qquad \forall t>0.\label{eqn:margin}
\end{align}
\end{condition}
This condition imposes that the unique $\Omega$-optimal policy can be estimated well via a plug-in estimator \citep{qian2011performance,Alex16}. For some generic $P\in\mathcal{P}$ and $\pi\in\Pi$, define $p_P(a|x):=P(A=a|X=x)$ and $D(\pi, P)(x,a,y^{\dag}) = \frac{\I\{a=\pi(x)\}}{p_P(a|x)}[y^\dag-s(a,x)]+s(\pi(x),x)-\Psi_{\pi}(P)$. We will use the shorthand $p_0:=p_{P_0}$ and $p_n:=p_{\widehat{P}_n}$. The following result characterizes the pathwise differentiability of $\Psi^*(\cdot)$ at $P_0$. 

\begin{lemma}\label{prop:path_diff}
Suppose that $\Psi_\pi$ and $\Omega_\pi$ are covariate-adjusted means for each $\pi\in\Pi$, the policy class $\Pi$ is unrestricted, and conditions~\ref{cond:margin_ZY} and \ref{cond:margin_Y} are satisfied. Then,
$\Psi^*$ is pathwise differentiable at $P_0$ relative to a nonparametric model with canonical gradient $D(\pi^*,P_0)$. 
\end{lemma}

We use the above result to argue that a one-step corrected estimator is efficient provided its influence function is equal to $D(\Pi^*,P_0)$. Consider some estimate $\hat{P}_n$ of the true distribution $P_0$. The one-step corrected estimator takes the form $\psi_{OS,n}:=\Psi_{\widehat{\pi}}(\hat{P}_n)+P_nD(\widehat{\pi},\hat{P}_n)$. For simplicity, when studying this estimator, we focus on the case where $\widehat{\pi}$ is a plug-in estimator of the $\Omega$-optimal policy, namely $\pi_{\widehat{P}_n}^*$. In principle, the policy estimator could be constructed using some other approach, such as outcome weighted learning \citep{zhao2012estimating}. 
Let $q_{b,n}(x)$ and $s_{b,n}(x)$ be some estimates for the conditional average treatment effects $q_{b,0}(x)$ and $s_{b,0}(x)$ respectively. Also, let $s_{n}(a,x)$ be some estimate for $s_0(a,x)$. Define the $L_r(P)$ norm of a generic function $f: \D\to\R$ as $\|f\|_{r,P}:=[\int_\D |f(t)|^r dP(t)]^{1/r}$. We first present some consistency conditions on these estimates. 

\begin{condition}[Consistent estimator of conditional average treatment effect on the primary outcome]\label{cond:4}
$\left\|q_{b, n}-q_{b, 0}\right\|_{\infty, P_{0}}^{1+\gamma/2}=o_{P_{0}}(n^{-1/2})$.
\end{condition}

\begin{condition}[Consistent estimator of conditional average treatment effect on the subsidiary outcome]\label{cond:5}
$\max_{a\in\{0,1\}} \left\{\left\|\frac{p_0(a \mid \cdot)}{p_{n}(a \mid \cdot)}-1\right\|_{2, P_{0}}\left\|s_{\widehat{P}_n}(a, \cdot)-s_{P_0}(a, \cdot)\right\|_{2, P_{0}}\right\}=o_{P_0}(n^{-1/2})$, where $s_P(a,x):=E_P[Y^\dagger\mid A=a,X=x]$.
\end{condition}
Condition~\ref{cond:5} is similar to Equation 15 in \cite{Alex16}. Discussion of this condition can be found in \cite{Alex16}.
The following theorem states that the one-step estimator is efficient. 

\begin{theorem}\label{prop:asymp_linear_psi}
Under Conditions \ref{cond:margin_ZY}, \ref{cond:margin_Y}, \ref{cond:4}, \ref{cond:5}, and also provided $D(\widehat{\pi},\hat{P}_n)$ falls in a $P_0$-Donsker class with probability tending to 1 and $\|D(\widehat{\pi},\hat{P}_n)-D(\pi^*,P_0)\|_{2,P_0}\overset{p}{\rightarrow} 0$, the one-step estimator $\psi_{OS,n}$ for $\hat{\pi}=\pi_{\hat{P}_n}^*$ is an asymptotically linear estimator of $\Psi^*(P_0)$ with influence function $D(\pi^*,P_0)$, in the sense that
\begin{align*}
\psi_{OS,n} - \Psi^*(P_0) = \frac{1}{n}\sum_{i=1}^n D(\pi^*,P_0)(X_i,A_i,Y_i^\dagger) + o_{P_0}(n^{-1/2}).
\end{align*}
Moreover, $\psi_{OS,n}$ is an asymptotically efficient estimator of $\psi_0$. 
\end{theorem}
The above can be used to construct Wald-type confidence intervals for $\psi_0$ of the form $\psi_{OS,n}\pm z_{1-\alpha/2}\sigma_n/\sqrt{n}$, where $z_{1-\alpha/2}$ is the $1-\alpha/2$ quantile of a standard normal random variable and $\sigma_n^2:= \frac{1}{n}\sum_{i=1}^n D(\widehat{\pi},\widehat{P}_n)(X_i,A_i,Y_i^\dagger)^2$.

The Donsker condition stated in the above theorem can be removed if cross-fitting is used \citep{schick1986asymptotically}. A 2-fold version of this approach first partitions the data in two halves. Then, it uses the first half of the data to learn $\hat{\pi}$ and uses the remaining data to construct an estimator for $\Psi_{\hat{\pi}}(P_0)$. The roles of the two halves are then swapped and the two estimators are subsequently averaged. Multi-fold versions of cross-fitting could also be used.

\section{Inference of a general functional without margin assumption}\label{sec:general_inference}

\subsection{Overview of the methods}\label{sec:overview}
The methods we present in this section are agnostic to whether Condition \ref{cond:margin_ZY} holds and, more generally, whether there are multiple $\Omega$-optimal policies. 
Because the parameter $\Psi^*$ considered in the previous section may not even be well-defined when there are multiple such policies, we instead focus on inferring about the range $[\psi_0^l, \psi_0^u]$ of possible $\Psi$-performances of $\Omega$-optimal policies. 
Unlike those in the previous section, the methods developed here critically rely on the policy class $\Pi$ being restricted --- in particular, being $P_0$-Donsker \citep{van1996weak} --- and this condition cannot be removed even if cross-fitting is employed (see Section~\ref{sec:unifCSVal} for a discussion). Also, in this section, we do not assume our performance criteria are covariate-adjusted means. Rather, they could take some other form, such as that of a covariate-adjusted median. 
In what follows we give an overview of our approach for inferring about $[\psi_0^l, \psi_0^u]$. 

Our proposed method consists of two stages. The first spends $\beta<\alpha$ error probability to construct a confidence set $\hat{\Pi}_\beta$ that contains the set of optimal policies $\Pi^*$ with probability tending to at least $1-\beta$. The second infers about the $\Psi$-performance of each remaining policy in this confidence set, returning a confidence interval for $[\psi_0^\ell,\psi_0^u]$ of the form
\begin{equation}
    \left[\inf_{\pi\in\hat{\Pi}_\beta}\left\{\hat{\psi}_\pi - \frac{\widehat{\kappa}_\pi z_{\alpha,\beta}}{n^{1/2}}\right\}, \sup_{\pi\in\hat{\Pi}_\beta}\left\{\hat{\psi}_\pi + \frac{\widehat{\kappa}_\pi z_{\alpha,\beta}}{n^{1/2}}\right\}\right],\label{eqn:asymp_int_two_stage}
\end{equation}
where $\hat{\psi}_\pi$ is some estimate for $\Psi_\pi(P_0)$, $z_{\alpha,\beta}$ corresponds to $1-(\alpha-\beta)/2$ quantile of normal distribution and $\widehat{\kappa}_\pi^2$ is an estimate of the asymptotic efficiency bound for estimating $\Psi_\pi(P_0)$. 
We provide a union bounding argument that shows that, under conditions, this confidence interval will cover $[\psi_0^\ell,\psi_0^u]$ with asymptotic probability $1-\alpha$.

The first-stage confidence set $\hat{\Pi}_\beta$ is constructed so that policies that perform poorly in terms of the primary performance metric are eliminated. In words, we maintain policies $\pi$ whose uniform upper confidence bound for $\Psi_\pi(P_0)$ is greater than the largest non-uniform lower confidence bound across all policies in the set. Figure~\ref{fig:filtration} shows an example of how the first-stage elimination is performed. More specifically, we define this set $\hat{\Pi}_\beta$ after the first-stage filtration as
\begin{align}
    \hat{\Pi}_\beta:= \left\{\pi\in\Pi : L_n\leq \hat{\omega}_\pi + \frac{\hat{\sigma}_\pi t_{\beta}}{n^{1/2}}\right\},\label{eqn:Pi1-beta}
\end{align}
where $\hat{\omega}_\pi$ is some estimate for $\Omega_\pi(P_0)$, $\widehat{\sigma}_\pi^2$ is an estimator of the asymptotic efficiency bound for estimating $\Omega_\pi(P_0)$, $L_n$ is an asymptotically valid $1-\beta/2$ lower bound for $\sup_{\pi\in\Pi}\Omega_\pi(P_0)$ (e.g., obtained via \citep{Alex16}), and $t_{\beta}$ is selected in such a way that $\{\hat{\omega}_\pi + \hat{\sigma}_\pi t_{\beta}/n^{1/2} : \pi\in\Pi\}$ is an asymptotically valid $1-\beta/2$ uniform upper confidence bound for $\{\Omega_\pi(P_0): \pi\in\Pi\}$, in the sense that $\Omega_\pi(P_0)\le \hat{\omega}_\pi + \hat{\sigma}_\pi t_{\beta}/n^{1/2}$ for all $\pi\in\Pi$ with probability tending to at least $1-\beta/2$ as $n$ goes to infinity. 

It may at first be surprising that, in constructing the confidence interval for $[\psi_0^\ell, \psi_0^u]$, the only place a uniform confidence bound is used is in the upper bound of \eqref{eqn:Pi1-beta}. 
Indeed, when we began studying this problem, the first approach that we considered was the same as that previously described, except with all confidence bounds replaced by uniform ones. 
In particular, $L_n$ was defined as the maximum over $\pi\in\Pi$ of a uniform lower confidence bound for the $\Omega$-value function and the minimal and maximal marginal confidence bounds in \eqref{eqn:asymp_int_two_stage} were also replaced by minimal and maximal uniform confidence bounds. 
However, after analyzing this method, we discovered that less uniformity was needed than we initially expected. 
Indeed, the uniformity in defining $L_n$ can be dropped since a simple union bounding argument shows that $L_n$ only needs to satisfy that it falls below the optimal $\Omega$-value with asymptotic probability at least $1-\beta/2$; while selecting the maximum of a uniform confidence band for the value function does satisfy such a property, developing such a lower bound is now a well-studied problem, and so less conservative approaches have been developed \citep{Alex16}. 
The uniformity on the second stage can be dropped via an intersection-union method argument \citep[Theorem~1 of][]{berger1996bioequivalence}, which we show can be applied since our interest concerns parameters defined as the maxima and minima over a set. 

\begin{figure}[tb]
    \centering
    \includegraphics[scale=0.33]{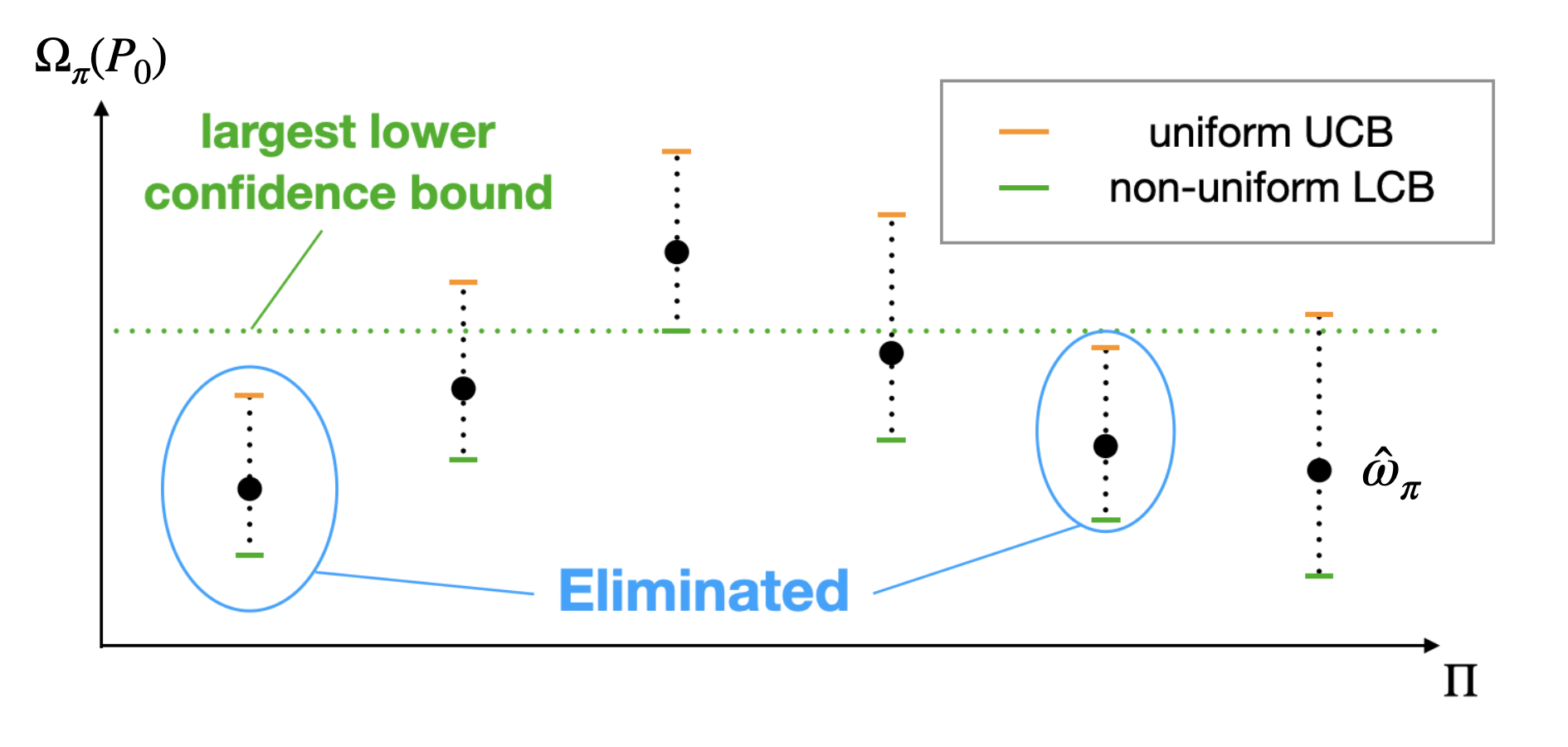}
    \caption{Example of first-stage elimination. Each black dot represents an estimate of $\Omega_\pi(P_0)$ and the horizontal bars denote the confidence bounds. Policies whose uniform upper confidence bound (UCB) is below the largest lower confidence bound (LCB) get eliminated.}
    \label{fig:filtration}
\end{figure}

As mentioned earlier, justifying the above approach relies on a union-bounding argument across the $\beta$ coverage error that could be made by the first-stage confidence interval in \eqref{eqn:Pi1-beta} and the $1-\alpha-\beta$ coverage error that could be made by the second-stage confidence interval in \eqref{eqn:asymp_int_two_stage}.
Relying on this union bound could result in unnecessarily wide confidence intervals, so we also present another two-stage method whose justification does not require a union bound. 
In the first stage of this approach we choose the quantiles $s_{\alpha}^\dagger$, $t_{\alpha}^\dagger$, and $u_{\alpha}^\dagger$ derived as extreme values of the joint distributions of estimators of $(\Omega_\pi(P_0))_{\pi\in\Pi}$ and $(\Psi_\pi(P_0))_{\pi\in\Pi}$ --- see Section~\ref{sec:joint} for details.  
Then we construct $\hat{\Pi}_\beta$ and the asymptotic interval the same ways as in \eqref{eqn:Pi1-beta} and \eqref{eqn:asymp_int_two_stage}, while replacing $t_{\beta}$ and $z_{\alpha,\beta}$ with $t_{\alpha}^\dagger$ and $s_{\alpha}^\dagger$, respectively. Given that $s_{\alpha}^\dagger$, $t_{\alpha}^\dagger$, and $u_{\alpha}^\dagger$ are constructed based on a joint distribution, we refer to this approach as the joint approach. Because of the avoidance of the union bound, the joint approach is expected to provide tighter confidence intervals in scenarios when the primary and subsidiary outcomes are strongly correlated. 

\subsection{A union bounding approach}\label{sec:unifCSVal}

In this subsection, we provide additional details and theoretical results about the union bounding approach. We first need the following condition for an estimator of $\{\Omega_\pi(P_0) : \pi\in\Pi\}$. In what follows, we let $\tilde{D}_\pi$ be the canonical gradient of $\Psi_\pi$ relative to a locally nonparametric model, $\sigma_\pi(P_0):= [P D_\pi(P_0)^2]^{1/2}$, and $\kappa_\pi(P_0):= [P \tilde{D}_\pi(P_0)^2]^{1/2}$. 
\begin{condition}[Uniform asymptotic linearity of estimators of $\Omega$-value and $\Psi$-value functions]\label{cond:asymp_linear_est}
The estimators $\{\hat{\omega}_\pi : \pi\in\Pi\}$ of $\{\Omega_\pi(P_0) : \pi\in\Pi\}$ and $\{\hat{\psi}_\pi : \pi\in\Pi\}$ of $\{\Psi_\pi(P_0) : \pi\in\Pi\}$ satisfy
\begin{align}
\sup_{\pi\in \Pi} \left[\hat{\omega}_\pi - \Omega_\pi(P_0) - P_n D_\pi(P_0)\right]= o_p(n^{-1/2}), \hspace{16pt} \sup_{\pi\in \Pi^*} \left[\hat{\psi}_\pi - \Psi_\pi(P_0) - P_n \tilde{D}_\pi(P_0)\right]= o_p(n^{-1/2}). \label{eq:uniformAL}
\end{align}
\end{condition}
These asymptotic linearity conditions can be established via consistency requirements similar to those in Condition~\ref{cond:5} and a Donsker condition (see Section 2.1 of \citep{luedtke2020performance}). Note that the latter equality in \eqref{eq:uniformAL} only requires uniformity over $\Pi^*$, rather than all of $\Pi$. Estimators satisfying \eqref{eq:uniformAL} can be derived via one-step estimation \citep{pfanzagl1982lecture}, targeted minimum loss-based estimation \citep{van2006targeted}, or double machine learning \citep{chernozhukov2018double}. We now provide some conditions on the $\Psi$-value function, the policy class $\Pi$ and necessary conditions for standard deviations and the primary outcome. 

\begin{condition}[Restricted policy class]\label{cond:restrict_policy_class}
The policy class $\Pi$ satisfies the following: 
\begin{enumerate}[label=(\arabic*)]
    \item $\Pi$ has a bounded uniform entropy integral (Chapter 2.5.1 of \citep{van1996weak}); 
    \item $\Pi$ is closed in $L^2(P_0)$, in the sense that, for all $\pi : \mathcal{X}\rightarrow\{0,1\}$, a $\Pi$-valued sequence $(\pi_k)_{k=1}^\infty$ converges to $\pi$ in $L^2(P_0)$ only if $\pi\in \Pi$;
    \item $\Pi^*$ is non-empty. 
\end{enumerate}
\end{condition}
Examples of such policy class $\Pi$ in $L^2(P_0)$ include classes of binary decision trees with fixed depths while noting that Condition~\ref{cond:restrict_policy_class} applies to more complicated and general policy classes. We then provide some conditions for the standard deviations and the smoothness of the $\Psi$-value function. 
\begin{condition}[Non-vanishing standard deviations and consistent estimators thereof]\label{cond:non_vanish_stdev}
The standard deviations satisfy the following conditions: $\inf_{\pi\in\Pi}\sigma_\pi(P_0)>0$, $\sup_{\pi\in\Pi}\sigma_\pi(P_0)<\infty$, $\inf_{\pi\in\Pi}\kappa_\pi(P_0)>0$, and $\sup_{\pi\in\Pi}\kappa_\pi(P_0)<\infty$. In addition, $\hat{\sigma}_\pi$ and $\hat{\kappa}_\pi$ are uniformly consistent estimators of $\sigma_\pi(P_0)$ and $\kappa_\pi(P_0)$.
\end{condition}
\begin{condition}[Smoothness of performance metric in policy]\label{cond:bounded_phi} 
    The map $\pi\mapsto \Psi_\pi(P_0)$ is continuous and, for all $\pi,\pi'\in\Pi$, $\norm{D_\pi-D_{\pi'}}_{L^2(P_0)}\leq C_2\norm{\pi-\pi'}_{L^2(P_0)}$ for some constant $C_2$.
\end{condition}

When $\Omega$ and $\Psi$ are covariate-adjusted mean functionals as in Section~\ref{sec:inf_under_margin_cond} and the primary and subsidiary outcomes are bounded, Condition~\ref{cond:bounded_phi} is necessarily true. Let $\mathcal{F}:=\{D_\pi(P_0)/\sigma_\pi(P_0) : \pi\in\Pi\}$ and $\tilde{\mathcal{F}}:=\{\tilde{f}_\pi:=\tilde{D}_\pi(P_0)/\kappa_\pi(P_0) : \pi\in\Pi\}$ denote the collections of canonical gradients that are standardized to have unit variance. Conditions~\ref{cond:non_vanish_stdev} and~\ref{cond:bounded_phi} play a crucial role in demonstrating that $\F$ and $\tilde{\F}$ are $P_0$-Donsker, which is required to validate the uniform confidence bands utilized in our union bounding approach. 

We now show that the confidence set $\hat{\Pi}_\beta$ contains the set of $\Omega$-optimal policies $\Pi^*$ with high probability asymptotically. Before presenting this result, we define the threshold $t_{\beta}$ used to define this confidence set in \eqref{eqn:Pi1-beta}. To this end, let $\{\mathbb{G} f: f\in\F\}$ be a mean-zero Gaussian process with a covariance function $(f_1,f_2)\mapsto Pf_1 f_2$. 
Then, $t_{\beta}$ is defined to be the $1-\beta/2$ quantile of $\sup_{f\in\mathcal{F}} \mathbb{G} f$ and Lemma \ref{prop:3.1} in the appendix shows that $\left\{\hat{\omega}_\pi \pm \frac{\hat{\sigma}_\pi t_{\beta}}{n^{1/2}} : \pi\in \Pi\right\}$ is an asymptotically valid uniform $\beta$-level confidence band for $\{\omega_\pi : \pi\in\Pi\}$.
\begin{lemma}[Asymptotic coverage of $\hat{\Pi}_\beta$]\label{lem:Pi1b}
If Conditions~\ref{cond:asymp_linear_est}, \ref{cond:restrict_policy_class}, and \ref{cond:non_vanish_stdev} hold, then $\limsup_n P\{\Pi^*\not\subseteq \hat{\Pi}_\beta\}\le \beta$.
\end{lemma}

The interval in \eqref{eqn:asymp_int_two_stage} uses the remaining $\alpha-\beta$ error probability to construct a confidence interval for the random quantity $\widehat{\mathcal{I}}_\beta:=[\inf_{\pi\in\widehat{\Pi}_\beta}\Psi_\pi(P_0), \sup_{\pi\in\widehat{\Pi}_\beta}\Psi_\pi(P_0)]$. On the event that $\Pi^*\subseteq \hat{\Pi}_\beta$, it is true that $\widehat{\mathcal{I}}_\beta\supseteq [\psi_0^\ell,\psi_0^u]$, and so any interval that covers $\widehat{\mathcal{I}}_\beta$ also covers $[\psi_0^\ell,\psi_0^u]$. A union bound then gives our result. Our findings are summarized in Theorem~\ref{thm:PsiPi}. 

\begin{theorem}[Asymptotic coverage of ${\rm CI}_n$] \label{thm:PsiPi}
Under Conditions \ref{cond:asymp_linear_est}, \ref{cond:restrict_policy_class}, \ref{cond:non_vanish_stdev}, and \ref{cond:bounded_phi}, for a fixed $\alpha\in (0,1)$ and any choice of $\beta\in (0,\alpha)$, the confidence interval ${\rm CI}_n$ as defined in \eqref{eqn:asymp_int_two_stage} satisfies $\liminf_{n\to\infty} \P(\{[\psi_0^\ell, \psi_0^u]\subseteq {\rm CI}_n\})\geq 1-\alpha$. 
\end{theorem}

Also, as indicated in \eqref{eqn:asymp_int_two_stage}, the width of ${\rm CI}_n$ is determined by a quantile of a standard normal random variable --- for example, when $\alpha=0.06$ and $\beta=0.01$, $z_{\alpha,\beta}\approx 1.96$. At first this may seem surprising, given that developing a uniform confidence band for $\{\Psi_{\pi}: \pi\in\Pi^*\}$ would require using a strictly larger scaling the standard error of $\widehat{\psi}_\pi$. 
However, our proof of Theorem~\ref{thm:PsiPi} shows that using this larger scaling is not necessary for the sake of developing a confidence interval for $[\psi_0^\ell, \psi_0^u]$. The key to this argument involves showing that, under Condition~\ref{cond:restrict_policy_class}, there exist $\pi^\ell$ and $\pi^u$ in $\Pi^*$ that attain the minimum and maximum $\Psi$-values, respectively. 
The existence of $\pi^\ell$ shows that the event where the lower bound of ${\rm CI}_n$ fails to cover $\psi_0^\ell:=\inf_{\pi\in \Pi^*}\Psi_\pi(P_0)$, intersected with $\Pi^*\subseteq \hat{\Pi}_\beta$, satisfies
\begin{align*}
    &\left\{\inf_{\pi\in \Pi^*} \Psi_\pi(P_0) <\inf_{\pi\in\hat{\Pi}_\beta}\left[\hat{\psi}_\pi - \hat{\kappa}_\pi z_{\alpha,\beta}/n^{1/2}\right], \Pi^*\subseteq \hat{\Pi}_\beta\right\}\subseteq \left\{\inf_{\pi\in \Pi^*} \Psi_\pi(P_0) <\inf_{\pi\in\Pi^*}\left[\hat{\psi}_\pi - \hat{\kappa}_\pi z_{\alpha,\beta}/n^{1/2}\right]\right\}\\
    &= \left\{\psi_{\pi^\ell} <\inf_{\pi\in\Pi^*}\left[\hat{\psi}_\pi - \hat{\kappa}_\pi z_{\alpha,\beta}/n^{1/2}\right]\right\} \subseteq \left\{\psi_{\pi^\ell} <\hat{\psi}_{\pi^\ell} - \hat{\kappa}_{\pi^\ell} z_{\alpha,\beta}/n^{1/2}\right\}.
\end{align*}
The event on the right corresponds to the case where a marginal $1-(\alpha-\beta)/2$-level lower Wald-type confidence interval fails to cover $\psi_{\pi^\ell}$, and so occurs with asymptotic probability $(\alpha-\beta)/2$ under reasonable conditions. 
In our proof of Theorem~\ref{thm:PsiPi}, we establish the result using a union bounding argument that combines this with a similar guarantee for the upper bound of $\mathrm{CI}_n$ and the fact that $\Pi^*\not\subseteq \hat{\Pi}_\beta$ happens with asymptotic probability at most $\beta$. 

Under additional conditions, our confidence interval for $[\psi_0^\ell, \psi_0^u]$ not only ensures asymptotically valid coverage but also attains an optimal $n^{-1/2}$ convergence rate. In this part, we restrict the performance metrics to covariate-adjusted means and propose a boundedness condition on the primary and subsidiary CATE functions.
\begin{condition}[Boundedness condition]\label{cond:boundedness}
There exists some $C_3<\infty$ such that for any $x\in\X$, we have $|s_{b,0}(x)|\leq C_3|q_{b,0}(x)|$. 
\end{condition}

In most ways, Condition~\ref{cond:boundedness} is relatively stronger than Condition~\ref{cond:margin_ZY}. Indeed, the limit of Condition~\ref{cond:margin_ZY} as $\zeta\rightarrow\infty$ corresponds to the condition that the subsidiary CATE is strictly less than a constant multiple of the primary CATE. Since Condition~\ref{cond:margin_ZY} allows for any $\zeta>2$, it puts a much weaker constraint on how the subsidiary outcome behaves for nearly optimal policies. There is one sense, however, in which Condition~\ref{cond:boundedness} is weaker than Condition~\ref{cond:margin_ZY}: it does not generally imply that the $\Omega$-optimal policy is necessarily unique. This is true because it allows for equality between subsidiary and primary CATEs, and so both could be zero on some set of positive probability. Though the optimal policy need not be unique when Condition~\ref{cond:boundedness} holds, it must still be true that all $\Omega$-optimal policies yield the same $\Psi$-value, and so in the following lemma we shall let $\psi_0=\psi_0^\ell=\psi_0^u$.

\begin{lemma}[$n^{-1/2}$ convergence rate of $\operatorname{CI}_n$ under conditions]\label{lem:conv_rate_CI} 
Assume that the performance metrics are covariate-adjusted means as in Section~\ref{sec:inf_under_margin_cond}, the unrestricted $\Omega$-optimal policy over all possible maps from $\X$ to $\{0,1\}$ is in $\Pi$, and $L_n=\sup_{\pi\in\Pi} \left[\hat{\omega}_\pi-\frac{\hat{\sigma}_\pi t_{\beta}}{n^{1/2}}\right]$ in~\eqref{eqn:Pi1-beta}. Then, under Conditions~\ref{cond:asymp_linear_est}, \ref{cond:restrict_policy_class}, \ref{cond:non_vanish_stdev}, and \ref{cond:boundedness}, with probability at least $1-2\beta$ asymptotically, the width of the confidence interval for $\psi_0$ is $O_p(n^{-1/2})$.
\end{lemma}

\subsection{A joint approach}\label{sec:joint}
We now formally describe our joint approach. 
Consider the mean-zero Gaussian process $\{\mathbb{G}f : f\in\mathcal{F}\cup\tilde{\mathcal{F}}\}$ with covariance function $(f_1,f_2)\mapsto Pf_1f_2$. 
Our joint approach is the same as the two-stage procedure from Section \ref{sec:unifCSVal}, except that we require a particular choice of $L_n$ and use cutoffs $(s_{\alpha}^\dagger, t_{\alpha}^\dagger,u_{\alpha}^\dagger)$ satisfying
\begin{equation}
    \inf_{\pi\in\Pi} \P\left\{\inf_{f\in\mathcal{F}}\mathbb{G}f\ge -t_{\alpha}^\dagger,\sup_{f\in\mathcal{F}}\mathbb{G}f\le s_{\alpha}^\dagger,\mathbb{G}\tilde{f}_\pi\ge -u_{\alpha}^\dagger, \mathbb{G}\tilde{f}_\pi\le u_{\alpha}^\dagger\right\}\ge 1-\alpha.\label{eqn:joint_cutoff}
\end{equation}
More specifically, we define the set $\hat{\Pi}^\dagger$ after the first-stage filtration as
\begin{align}
    \widehat{\Pi}^\dagger:= \left\{\pi\in\Pi : \sup_{\pi\in\Pi} \left[\hat{\omega}_\pi-\frac{\hat{\sigma}_\pi s_{\alpha}^\dagger}{n^{1/2}}\right]\leq \hat{\omega}_\pi + \frac{\hat{\sigma}_\pi t_{\alpha}^\dagger}{n^{1/2}}\right\}.\label{eqn:Pihat}
\end{align}
Here we choose $L_n$ to be the uppermost point of a uniform lower confidence band for $\{\Omega_{\pi}(P_0):\pi\in\Pi\}$ with level $\beta^\dagger:=\P\{\sup_{\pi\in\Pi}\mathbb{G}f_\pi > s_{1-\alpha}^\dagger\}<\alpha$. Note that in the union bounding approach, $\beta^\dagger=\beta$, while here $\beta^\dagger$ is implicitly defined through the joint cutoff \eqref{eqn:joint_cutoff}. 
The resulting confidence interval is stated in Theorem~\ref{thm:joint_CI}. 

\begin{theorem}\label{thm:joint_CI}
Under Conditions \ref{cond:asymp_linear_est}, \ref{cond:restrict_policy_class}, \ref{cond:non_vanish_stdev}, and \ref{cond:bounded_phi}, assuming the cutoffs $(s_{\alpha}^\dagger,t_{\alpha}^\dagger,u_{\alpha}^\dagger)$ satisfy \eqref{eqn:joint_cutoff}, it holds that $\liminf_{n\to\infty} \P(\{[\psi_0^\ell, \psi_0^u]\subseteq {\rm CI}_n^\dagger\})\geq 1-\alpha$, where 
$${\rm CI}_n^\dagger:=\left[\inf_{\pi\in\hat{\Pi}^\dagger}\left\{\hat{\psi}_\pi-\frac{\hat{\kappa}_\pi u_{\alpha}^\dagger}{n^{1/2}}\right\},\sup_{\pi\in\hat{\Pi}^\dagger}\left\{\hat{\psi}_\pi+\frac{\hat{\kappa}_\pi u_{\alpha}^\dagger}{n^{1/2}}\right\}\right].$$ 
\end{theorem}
There are many possible choices of $(s_{\alpha}^\dagger,t_{\alpha}^\dagger,u_{\alpha}^\dagger)$ that satisfy \eqref{eqn:joint_cutoff}. To select among these, we could choose the triple $(s_{\alpha}^\dagger,t_{\alpha}^\dagger,u_{\alpha}^\dagger)$ that provides the tightest confidence interval from this collection, resulting in what we refer to as an optimized joint method. This optimized $(s_{\alpha}^\dagger,t_{\alpha}^\dagger,u_{\alpha}^\dagger)$ is justified since, for any choice of $(s_{\alpha}^\dagger,t_{\alpha}^\dagger,u_{\alpha}^\dagger)$ satisfying \eqref{eqn:joint_cutoff}, the confidence interval $\rm{CI}_n^\dagger$ has valid coverage. For any $\beta$, this optimized joint method yields a provably tighter confidence interval than the union bounding method that uses the same choice of $L_n$ as in the left-hand side of \eqref{eqn:Pihat}. However, it is possible that the joint approach could potentially result in a wider confidence band in the first stage with the use of an alternative lower confidence bound for the $\Omega$-optimal value, such as the one introduced in \cite{Alex16}. 
In practice, the optimized choice of $(s_{\alpha}^\dagger,t_{\alpha}^\dagger,u_{\alpha}^\dagger)$ is unknown, but it can be approximated via a multiplier bootstrap --- see Appendix~\ref{sec:pseudo_mb} for details. Though our theorem focuses on a fixed and known triple $(s_{\alpha}^\dagger,t_{\alpha}^\dagger,u_{\alpha}^\dagger)$, adapting it to allow for the use of an estimated triple with an in-probability limit would be straightforward.  

The cutoff in \eqref{eqn:joint_cutoff} considers the joint event regarding $\GG \tilde{f}$ and $\GG f$ for $f\in\mathcal{F}$ and $\tilde{f}\in\tilde{\mathcal{F}}$, thereby avoiding the use of the union bound required by the approach in Section~\ref{sec:unifCSVal}. The tightness of this union bound relies on whether the event that $\Pi^*$ is contained in the first stage policy set, namely $\{\Pi^*\subseteq \hat{\Pi}_\beta\}$, and the event that $[\psi_0^{\ell}, \psi_0^u]$ is contained in the second stage confidence interval are disjoint. Of course, when these events are fully disjoint, the union bound will be tight. 
When they are independent, the (asymptotic) probability that both events occur is $\beta(\alpha-\beta)$, which will be small for choices of $\alpha$ and $\beta$ commonly used in practice. Hence, the union bound will only be slightly loose in these cases. Finally, when the events fully overlap, the union bound will be as loose as possible. 
These scenarios can be better understood by relating them to primary and subsidiary outcomes. Generally, the dependence or independence between the events is likely to correlate with the extent to which primary and subsidiary outcomes depend on each other. The events tend to be independent when primary and subsidiary outcomes are independent, and dependent otherwise.

\section{Numerical experiment}\label{sec:simulation}

\subsection{A 1D simulation}\label{sec:1D_sim} 
We conduct simulation studies to evaluate the length and coverage of $1-\alpha$ confidence intervals for bounds on a mean subsidiary outcome, $[\psi_0^\ell,\psi_0^u]$. Our first set of simulations focuses on a 1-dimensional threshold policy class, denoted as $\Pi=\{\mathbf{1}_{[a,\infty)}: a \in [-1, 1]\}$. We compare the confidence intervals from four approaches. The first is the union bounding approach described in Section~\ref{sec:unifCSVal}, denoted as \textsf{union}. The second is the joint approach described in Section~\ref{sec:joint}, denoted as \textsf{joint}. The third is the one-step estimator approach described in Section~\ref{sec:inf_under_margin_cond}, denoted as \textsf{one-step}. To ensure that this approach applies, we design our scenarios so that the optimal policy for the unrestricted policy class lies in the threshold class $\Pi$. Consequently, in our simulation study, an estimate of the optimal policy in $\Pi$ also estimates the optimal policy in the unrestricted class. 
The fourth is a one-step estimator with sample splitting, denoted as \textsf{os-split}. This approach is the same as \textsf{one-step}, except that we obtain an estimate $\widehat{\pi}_1$ of the $\Omega$-optimal policy using only half of the data, and construct a Wald-type confidence interval for $\Psi_{\widehat{\pi}_1}(P_0)$ using the other half. 
Last, we present an oracle method, denoted as \textsf{oracle}, that knows the specific $\Omega$-optimal policies that provide the upper and lower bounds, $\psi_0^u$ and $\psi_0^\ell$. The oracle method uses precisely those policies and construct a Wald-type confidence interval for $[\psi_0^\ell, \psi_0^u]$. Since we have no hope of getting optimal policies \textit{a priori}, the oracle method cannot be used in practice.  

We examine three distinct scenarios with an illustration of the $\Omega$ and $\Psi$ values of each policy under various scenarios in the three panels in Figure~\ref{fig:scenarios}. The left panel describes the situation where the set of $\Omega$-optimal policies, $\Pi^*$, is not unique. In this scenario, $\Pi^*=\{\mathbf{1}_{[a,\infty)}: a \in [-0.5, 0]\}$, and the margin condition (Condition \ref{cond:margin_ZY}) is not satisfied for any $\zeta$. The middle panel describes the situation where $\Pi^*$ is unique the margin condition is satisfied for any $\zeta>2$, as we can see that when $\pi$ is around the optimal policy, $q_{b,0}(X)$ varies much faster than $s_{b,0}(X)$. The right panel describes the situation where $\Pi^*$ is unique but the margin condition is not satisfied for any $\zeta$, as we can see that as $X$ varies, both $q_{b,0}(X)$ and $s_{b,0}(X)$ vary linearly. 

\begin{figure}[!htb]
    \centering
    \hspace{-31pt}\includegraphics[scale=0.236]{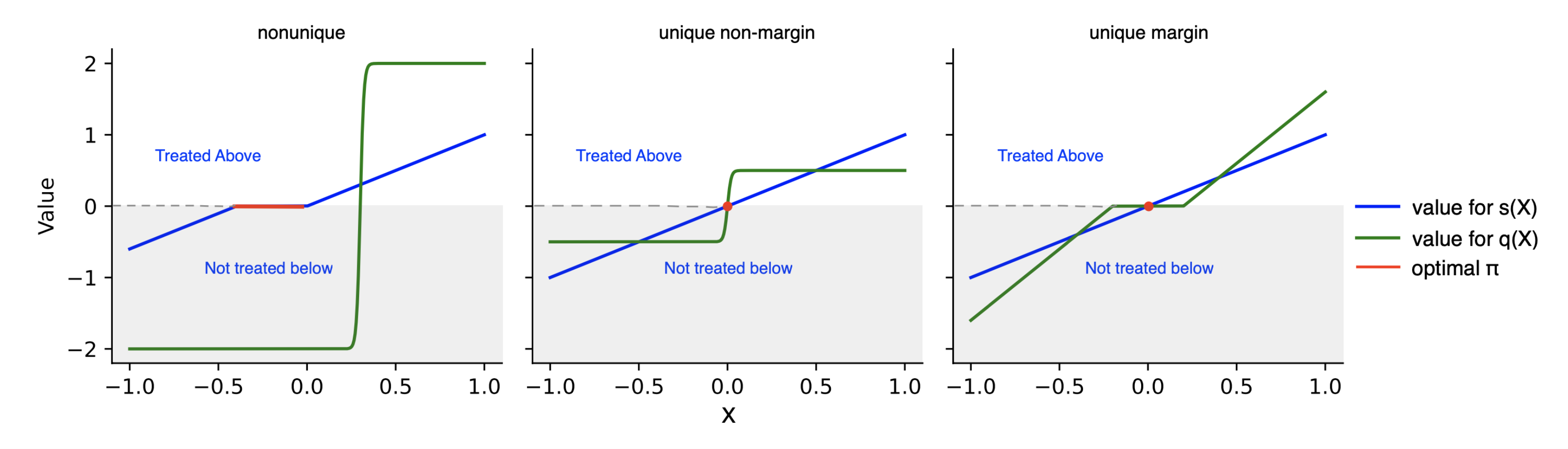}
    \includegraphics[scale=0.3]{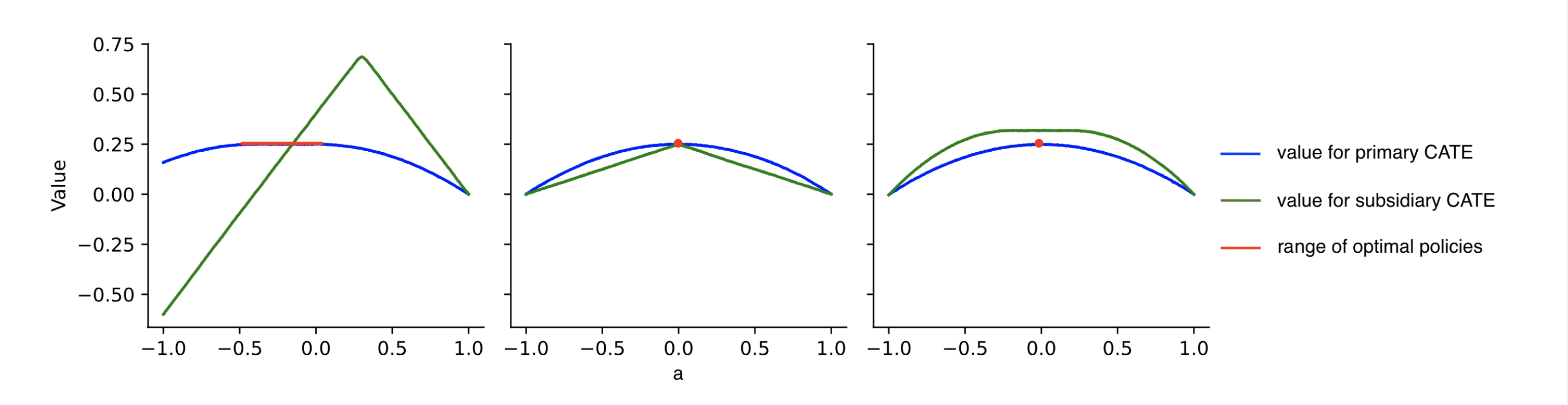}
    \caption{An illustration of $s_{b,0}(X)$- and $q_{b,0}(X)$-value and $\Omega$- and $\Psi$-value for a 1-dimensional threshold policy class $\Pi=\{\mathbf{1}_{[a,\infty)}: a \in [-1, 1]\}$ under different scenarios: the optimal policy for the primary outcome is nonunique, the optimal policy for the primary outcome is unique while the primary and subsidiary outcomes are correlated, and the optimal policy for the primary outcome is unique while the primary and subsidiary outcomes are not so correlated. The top figure represents $\Omega$- and $\Psi$-value, while the bottom figure represents $s_{b,0}(X)$- and $q_{b,0}(X)$-value.}
    \label{fig:scenarios}
\end{figure}

For each scenario, we consider sample sizes $n$ of 500 and 5000. To generate the set of policies, we construct a fine grid $(a_1,\cdots,a_N)$ for $N=10^5$ over $[-1,1]$ and denote the set of policy as $\Pi_N=\{\mathbf{1}_{[a_i,\infty)}: i\in[N]\}$. We use 1000 multiplier bootstrap replicates to estimate the supremum and infimum in generating the cutoffs. We let $\alpha=0.05$ when constructing confidence intervals and use 1000 Monte Carlo replications to compute their coverage of the true interval $[\psi_0^\ell, \psi_0^u]$ as well as approximate their average widths. We estimate the conditional probability $p(a|x)$ via a kernel density estimator as implemented in the \textsf{sklearn} package and the conditional probabilities $p(y|1,x)$ and $p(y|0,x)$ using gradient boosted trees as implemented in the \textsf{xgboost} package, both with the default settings. 
The Python code to reproduce the simulations is available at https://github.com/zhaoqil/EstimationSubsidiary. 

Table~\ref{tab:CI_coverage} shows the coverages and the widths of confidence intervals of $[\psi_0^{\ell}, \psi_0^u]$ for different scenarios and different methods. We can see the the one-step estimator fails to provide a nominal coverage when the margin condition (Condition~\ref{cond:margin_ZY}) is not satisfied. The other two methods produce similar coverages. We compare the confidence intervals with an oracle confidence interval, which is a lower bound on the width of any valid $1-\alpha$ confidence interval, and calculate the relative widths. 
We can see that the joint and union bounding methods generate confidence intervals about 2.3 times and 2.1 times as wide as the oracle confidence interval when the optimal policy is non-unique and unique, respectively. These results show that although our methods are conservative, they are relatively successful in maintaining a narrow confidence interval. In contrast, the one-step estimator produces a confidence interval that is about the same width as the oracle confidence interval, but it fails to provide valid coverage when the margin condition fails. Table~\ref{tab:CI_coverage_large} provides coverages and confidence interval widths with a larger sample size of 5000. In the non-unique setting, since there are multiple optimal policies for the primary outcome, $[\psi_0^\ell, \psi_0^u]$ will be an interval with some length. In our setting, we can see from the lower-left plot of Figure~\ref{fig:scenarios} that the length of $[\psi_0^\ell, \psi_0^u]$ is about 0.5, so any valid confidence interval for $[\psi_0^\ell, \psi_0^u]$ must have at least that length. Comparing the widths in Table~\ref{tab:CI_coverage} and \ref{tab:CI_coverage_large}, we can see that both the union bounding method and the joint method produce confidence intervals approaching that limit. In the setting where $\Omega$-optimal policy is unique, the widths of the confidence intervals for all methods approach zero as $n$ goes to infinity. 

\begin{table}[tb]
\centering
\begin{tabular}{|l|llll|lllll|}
\hline
                  & \multicolumn{4}{c|}{coverage}                                      & \multicolumn{5}{c|}{ width}                                                                    \\ \hline
                  & \multicolumn{1}{l|}{union} & \multicolumn{1}{l|}{joint} & \multicolumn{1}{l|}{one-step} & os-split & \multicolumn{1}{l|}{union} & \multicolumn{1}{l|}{joint} & \multicolumn{1}{l|}{one-step} & \multicolumn{1}{l||}{os-split} & oracle \\ \hline
non-unique        & \multicolumn{1}{l|}{1.000}   & \multicolumn{1}{l|}{1.000} & \multicolumn{1}{l|}{0.000}   & 0.000      & \multicolumn{1}{l|}{1.549} & \multicolumn{1}{l|}{1.538} & \multicolumn{1}{l|}{0.240} & \multicolumn{1}{l||}{0.317}    & 0.668  \\ \hline
unique non-margin & \multicolumn{1}{l|}{0.980}  & \multicolumn{1}{l|}{0.980} & \multicolumn{1}{l|}{0.812}  & 0.751    & \multicolumn{1}{l|}{0.148} & \multicolumn{1}{l|}{0.143} & \multicolumn{1}{l|}{0.068} & \multicolumn{1}{l||}{0.089}    & 0.068  \\ \hline
unique margin     & \multicolumn{1}{l|}{0.978} & \multicolumn{1}{l|}{0.981} & \multicolumn{1}{l|}{0.949} & 0.953    & \multicolumn{1}{l|}{0.149} & \multicolumn{1}{l|}{0.144} & \multicolumn{1}{l|}{0.074} & \multicolumn{1}{l||}{0.108}    & 0.074  \\ \hline
\end{tabular}
\caption{Coverage and width of $[\psi_0^\ell, \psi_0^u]$ for different scenarios with sample size $n=500$}
\label{tab:CI_coverage}
\end{table}

\begin{table}[tb]
\centering
\begin{tabular}{|l|llll|lllll|}
\hline
                  & \multicolumn{4}{c|}{coverage}                                      & \multicolumn{5}{c|}{ width}                                                                    \\ \hline
                  & \multicolumn{1}{l|}{union} & \multicolumn{1}{l|}{joint} & \multicolumn{1}{l|}{one-step} & os-split & \multicolumn{1}{l|}{union} & \multicolumn{1}{l|}{joint} & \multicolumn{1}{l|}{one-step} & \multicolumn{1}{l||}{os-split} & oracle \\ \hline
non-unique        & \multicolumn{1}{l|}{1.000}   & \multicolumn{1}{l|}{1.000} & \multicolumn{1}{l|}{0.000}   & 0.000      & \multicolumn{1}{l|}{1.091} & \multicolumn{1}{l|}{1.061} & \multicolumn{1}{l|}{0.061} & \multicolumn{1}{l||}{0.096}    & 0.561  \\ \hline
unique non-margin & \multicolumn{1}{l|}{0.981}  & \multicolumn{1}{l|}{0.986} & \multicolumn{1}{l|}{0.810}  & 0.734    & \multicolumn{1}{l|}{0.036} & \multicolumn{1}{l|}{0.035} & \multicolumn{1}{l|}{0.017} & \multicolumn{1}{l||}{0.027}    & 0.016  \\ \hline
unique margin     & \multicolumn{1}{l|}{0.983} & \multicolumn{1}{l|}{0.989} & \multicolumn{1}{l|}{0.946} & 0.949    & \multicolumn{1}{l|}{0.040} & \multicolumn{1}{l|}{0.036} & \multicolumn{1}{l|}{0.023} & \multicolumn{1}{l||}{0.037}    & 0.023  \\ \hline
\end{tabular}
\caption{Coverage and width of $[\psi_0^\ell, \psi_0^u]$ for different scenarios with sample size $n=5000$}
\label{tab:CI_coverage_large}
\end{table}

\subsection{A 3D simulation}

We also added a scenario where we have a 3D policy and the optimal policy is unique. The policy class is a restricted tree class, denoted as $\Pi=\{x\mapsto \mathbf{1}\{x\geq a_1, x\geq a_2, x\geq a_3\}: a_1, a_2, a_3\in[-1, 1]\}\}$. We design our scenario so that the optimal policy for the unrestricted policy class lies in the tree class. We compare the outcome interval from three approaches: \textsf{union}, \textsf{joint}, and \textsf{one-step}. The method \textsf{os-split} provides a wider interval while having a worse coverage than \textsf{one-step} in 1D simulation results, so we drop it from the simulation. For each scenario, we consider a sample size $n$ of 500. We again use 1000 multiplier bootstrap replicates to estimate the supremum and infimum. In this scenario, instead of generating a fine grid and computing the maximum over the grid, we use the \textsf{nlopt} package to numerically approximate the maximum. We let $\alpha=0.05$ and use 500 Monte Carlo replications to compute the coverage and approximate the average confidence interval widths. Table~\ref{tab:3D} shows the results. The joint methods achieves slightly shorter widths in this setting (5-6\%), and the results are otherwise similar to those from Section~\ref{sec:1D_sim}.

\begin{table}[tb]
\centering
\begin{tabular}{|l|lll|llll|}
\hline
          & \multicolumn{3}{c|}{coverage}                                      & \multicolumn{4}{c|}{width}                                                                    \\ \hline
          & \multicolumn{1}{l|}{union} & \multicolumn{1}{l|}{joint} & one-step & \multicolumn{1}{l|}{union} & \multicolumn{1}{l|}{joint} & \multicolumn{1}{l|}{one-step} & oracle \\ \hline
3D margin & \multicolumn{1}{l|}{0.970} & \multicolumn{1}{l|}{0.948} & 0.940    & \multicolumn{1}{l|}{0.199} & \multicolumn{1}{l|}{0.186} & \multicolumn{1}{l|}{0.124}    & 0.124  \\ \hline
3D non-margin & \multicolumn{1}{l|}{1.000} & \multicolumn{1}{l|}{0.988} & 0.594    & \multicolumn{1}{l|}{0.185} & \multicolumn{1}{l|}{0.175} & \multicolumn{1}{l|}{0.092}    & 0.092  \\ \hline
\end{tabular}
\caption{Coverage and width for 3D policy class with sample size $n=500$}
\label{tab:3D}
\end{table}

\section{Discussion}

The problem studied in existing works aiming to infer about the optimal value of an optimal rule can be viewed as a special case of our setup, where the subsidiary and primary outcomes coincide. In these cases, our two-stage approaches provide ways to make inference without the margin condition considered in such works \cite{qian2011performance,Alex16}. Instead, we need uniform asymptotic linearity for the value functions and an appropriately restricted policy class. The margin condition could fail if the subsidiary metric varies too much across the set of policies that are nearly optimal for the primary metric \cite{luedtke2020performance}. However, if the policy class is Donsker and the estimator is established via debiased machine learning, the uniform asymptotic linearity condition will be plausible even when a margin condition does not hold. 

In our numerical experiments, our union bounding and joint approaches produced valid confidence intervals, even if they were somewhat conservative. 
Under margin conditions, these intervals attain a parametric $n^{-1/2}$ rate, matching those based on an efficient one-step estimator, although with a less favorable leading constant. 
However, when the margin conditions fail, intervals based on the one-step estimator fail to achieve valid coverage. 
In future research, it would be interesting to develop an adaptive procedure that is leading-constant-optimal under margin conditions and, even without them, can produce intervals that provide valid coverage.

As for other future work, it is worth exploring methods for inferring subsidiary metrics using observations from adaptive experiments, which are non-independent but have a martingale structure. Observations from longitudinal settings could also be considered. Additionally, one could examine simultaneous inference for multiple subsidiary metrics rather than one. 

\section*{Acknowledgements}

This work was supported by the National Institutes of Health under award number DP2-LM013340 and the National Science Foundation under award number DMS-2210216.

\bibliographystyle{unsrtnat}
\bibliography{references}

\appendix

\section{Proofs for Section~\ref{sec:inf_under_margin_cond}}\label{sec:proof_sec4}

Let $Q_{X,0}$ be the marginal distribution of $X$ under $P_0$, and let $Q_{Y^*,0}$ and $Q_{Y^\dagger,0}$ be respectively the conditional distribution of $Y^*$ and $Y^\dagger$ given $A$, $X$ under $P_0$. Let $\{P_\epsilon:\epsilon\in\R\}\subset \M$ be a parametric submodel that is such that $P_\epsilon=P_0$ when $\epsilon=0$. This submodel is defined so that the marginal distribution of $X$ and the conditional distributions of $Y^\dagger$ and $Y^*$ given $(A,X)$ satisfy
\begin{align}
d Q_{X, \epsilon}(x)&=(1+\epsilon S_{X}(x)) d Q_{X, 0}(x), \text { where } \mathbb{E}_{0}\left[S_{X}(x)\right]=0 \text { and } \sup _{x}\left|S_{X}(x)\right|\leq m<\infty,\label{eqn:S_x} \\
d Q_{Y^\dagger, \epsilon}(z \mid a,x) &=\left(1+\epsilon S_{Y^\dagger}(z \mid a,x)\right) d Q_{Y^\dagger, 0}(z \mid a,x) \label{eqn:S_Z},\\
& \text { where } \mathbb{E}_{0}\left[S_{Y^\dagger} \mid A, X\right]=0 \textnormal{ $P_0$-a.s.} \text { and } \sup _{x, a, z}\left|S_{Y^\dagger}(z \mid a, x)\right|<\infty,\textnormal{ and }\nonumber \\
d Q_{Y^*, \epsilon}(y \mid a,x) &=\left(1+\epsilon S_{Y^*}(y \mid a,x)\right) d Q_{Y^*, 0}(y \mid a,x) \label{eqn:S_Y}\\
& \text { where } \mathbb{E}_{0}\left[S_{Y^*} \mid A, X\right]=0 \textnormal{ $P_0$-a.s.} \text { and } \sup _{x, a, y}\left|S_{Y^*}(y \mid a, x)\right|<\infty.\nonumber
\end{align}
We let $q_{b,\epsilon}(x)=q_b(P_\epsilon)(x)$ and $s_{b,\epsilon}(x)=s_b(P_\epsilon)(x)$. 
\begin{proof}[Proof of Lemma~\ref{prop:path_diff}]

Note that $\pi_P^*(x)=\I\{q_b(P)(x)>0\}$ for all $x\in\X$. Following the idea of the proof of Theorem 3 in \cite{Alex16}, we observe that 
\[\begin{aligned}
\Psi^*(P)-\E_{P} \E_{P}[Y^\dagger \mid A=0, X] &=\E_{P}\left[\pi_P^{*}(X) s_{b}(P)(X)\right].
\end{aligned}\]
By a telescoping argument, 
\begin{align}
    \Psi^*(P_\epsilon)-\Psi^*(P_0) &= \E_{P_\epsilon}\E_{P_\epsilon}[Y^\dagger|A=\pi_{P_\epsilon}^*(X),X]-\E_{P_0}\E_{P_0}[Y^\dagger|A=\pi^*(X),X]\nonumber\\
    &= \E_{P_\epsilon}\E_{P_\epsilon}[Y^\dagger|A=\pi_{P_\epsilon}^*(X),X]-\E_{P_\epsilon}\E_{P_\epsilon}[Y^\dagger|A=\pi^*(X),X]\nonumber\\
    &\quad +\E_{P_\epsilon}\E_{P_\epsilon}[Y^\dagger|A=\pi^*(X),X]-\E_{P_0}\E_{P_0}[Y^\dagger|A=\pi^*(X),X]\nonumber\\
    &= \E_{P_\epsilon}[(\I(q_{b,\epsilon}>0)-\I(q_{b,0}>0))\cdot s_{b,\epsilon}]+\Psi_{\pi^*}(P_\epsilon)-\Psi_{\pi^*}(P_0).\label{eqn:decomp_psi}
\end{align}
It is known that for a fixed $\pi$, $\Psi_\pi$ is pathwise differentiable with gradient $D(\pi,P_0)$. We shall now show that the first term is $o(\epsilon)$. Letting $B_1:=\{x\in\X:q_{b,0}(x)=0\}$, 
we have
\begin{align*}
&\E_{P_{\epsilon}}\left[\left(I\left(q_{b, \epsilon}>0\right)-I\left(q_{b, 0}>0\right)\right) s_{b, \epsilon}\right] \\
&=\int_{\mathcal{X} \backslash B_{1}}\left(I\left(q_{b, \epsilon}>0\right)-I\left(q_{b, 0}>0\right)\right) s_{b, \epsilon} d Q_{X, \epsilon}+\int_{B_{1}}\left(I\left(q_{b, \epsilon}>0\right)-I\left(q_{b, 0}>0\right)\right) s_{b, \epsilon} d Q_{X, \epsilon}.
\end{align*}
Under Condition~\ref{cond:margin_ZY}, we know that $\Pr_0(q_{b,0}(X)\ne 0)=1$, so the second term is zero. Then we aim to show that the first term is $o(|\epsilon|)$. Note that 
\begin{align*}
    \left|\int_{\mathcal{X} \backslash B_{1}}\left(I\left(q_{b, \epsilon}>0\right)-I\left(q_{b, 0}>0\right)\right) s_{b, \epsilon} d Q_{X, \epsilon} \right|&\leq \int_{\mathcal{X} \backslash B_{1}}\left|\left(I\left(q_{b, \epsilon}>0\right)-I\left(q_{b, 0}>0\right)\right) s_{b, \epsilon}\right| d Q_{X, \epsilon} \\
    &\leq \int_{\mathcal{X} \backslash B_{1}}I\left(\bigabs{q_{b, 0}}<\bigabs{q_{b, \epsilon}-q_{b, 0}}\right) \left|s_{b, \epsilon}\right| d Q_{X, \epsilon}
\end{align*}
by looking at the sign of $q_{b, \epsilon}$ and $q_{b, 0}$. Also, 
\begin{align*}
q_{b, \epsilon}(x) &= \int y\left(d Q_{Y^*, \epsilon}(y \mid A=1, X=x)-d Q_{Y^*, \epsilon}(y \mid A=0, X=x)\right) \\
&=q_{b, 0}(x)+\epsilon\left(\E_{0}\left[Y^* S_{Y^*}(Y^* \mid 1, X) \mid A=1, X=x\right]-\E_{0}\left[Y^* S_{Y^*}(Y^* \mid 0, X) \mid A=0, X=x\right]\right)\\
&= q_{b, 0}(x)+\epsilon \bar{h}(x)
\end{align*}
where \[\bar{h}(x)=\E_0[Y^*S_{Y^*}(Y^*|1,X)|A=1,X=x]-\E_0[Y^*S_{Y^*}(Y^*|0,X)|A=0,X=x].\]
Similarly, $s_{b,\epsilon}(x)=s_{b,0}(x)+\epsilon\cdot \tilde{h}(x)$ where \[\tilde{h}(x)=\E_0[Y^\dagger S_{Y^\dagger}(Y^\dagger|1,X)|A=1,X=x]-\E_0[Y^\dagger S_{Y^\dagger}(Y^\dagger|0,X)|A=0,X=x].\]
Note that $\tilde{h}$ and $\bar{h}$ are uniformly bounded since $Y^*$, $Y^\dagger$, $S_{Y^*}$, and $S_{Y^\dagger}$ are bounded. Let $H=\max\{\sup_{x}|\bar{h}(x)|,\\ \sup_{x}|\tilde{h}(x)|\}$. Therefore, 
\begin{align*}
    \int_{\mathcal{X} \backslash B_{1}}I\left(\bigabs{q_{b, 0}}<\bigabs{q_{b, \epsilon}-q_{b, 0}}\right) \left|s_{b, \epsilon}\right| d Q_{X, \epsilon}&\leq \int_{\mathcal{X} \backslash B_{1}}I\left(\bigabs{q_{b, 0}}<H|\epsilon|\right) \bigsmile{\left|s_{b, 0}\right|+H|\epsilon|} d Q_{X, \epsilon}\\
    &\leq (1+m|\epsilon|)\int_{\mathcal{X} \backslash B_{1}}I\left(\bigabs{q_{b, 0}}<H|\epsilon|\right) \bigsmile{\left|s_{b, 0}\right|+H|\epsilon|} d Q_{X, 0}\\
    &= (1+m|\epsilon|)\int_{\mathcal{X} \backslash B_{1}}I\left(0<\bigabs{q_{b, 0}}<H|\epsilon|\right) \bigsmile{\left|s_{b, 0}\right|+H|\epsilon|} d Q_{X, 0}.
\end{align*}
Denote $\tilde{\X}=\X\setminus B_1$. Under the first condition, define the set \[B_{2,t}=\{x\in\tilde{\X}:|s_{b,0}(x)|< Ct^{-1}|q_{b,0}(x)|\}.\] Then 
\begin{align*}
    &\int_{\mathcal{X} \backslash B_{1}}I\left(\bigabs{q_{b, 0}}<H|\epsilon|\right) \bigsmile{\left|s_{b, 0}\right|+H|\epsilon|} d Q_{X, 0}\\
    &=\int_{\tilde{\mathcal{X}}}I\left(0<\bigabs{q_{b, 0}}<H|\epsilon|\right) \bigsmile{\left|s_{b, 0}\right|+H|\epsilon|} d Q_{X, 0}\\
    &= \int_{B_{2,t}}I\left(0<\bigabs{q_{b, 0}}<H|\epsilon|\right) \bigsmile{\left|s_{b, 0}\right|+H|\epsilon|} d Q_{X, 0}+\int_{\tilde{\X}\setminus B_{2,t}}I\left(0<\bigabs{q_{b, 0}}<H|\epsilon|\right) \bigsmile{\left|s_{b, 0}\right|+H|\epsilon|} d Q_{X, 0}.
\end{align*}
On one hand, note that for $x\in B_{2,t}$ and under the fact that $|q_{b,0}(x)|\leq H|\epsilon|$ we have $|s_{b}(x)|\leq CHt^{-1}|\epsilon|$.  define $C_2$ such that $P_0(0<|q_{b,0}(X)|<t)\leq C_2t^{\gamma}$ for any $t>0$, the first term 
\begin{align}
    \int_{B_{2,t}}I\left(0<\bigabs{q_{b, 0}}<H|\epsilon|\right) \bigsmile{\left|s_{b, 0}\right|+H|\epsilon|} d Q_{X, 0}&\leq \int_{B_{2,t}}I\left(0<\bigabs{q_{b, 0}}<H|\epsilon|\right) \bigsmile{CHt^{-1}|\epsilon|+H|\epsilon|} d Q_{X, 0}\nonumber\\
    &\leq \bigsmile{CHt^{-1}|\epsilon|+H|\epsilon|}P_0\left(0<\bigabs{q_{b, 0}(X)}<H|\epsilon|\right)\nonumber\\
    &\leq \bigsmile{Ct^{-1}|\epsilon|+H|\epsilon|}  C_2(H|\epsilon|)^{\gamma}\label{eqn:first_term}
\end{align}
for $t<1$. For the second term, let $C_3:=\sup_{x}|s_{b,0}(x)|$, we have
\begin{align*}
    &\int_{\tilde{\X}\setminus B_{2,t}}I\left(0<\bigabs{q_{b, 0}}<H|\epsilon|\right) \bigsmile{\left|s_{b, 0}\right|+H|\epsilon|} d Q_{X, 0}\\
    &\leq (C_3+H|\epsilon|) P_0(0<|s_{b, 0}(X)|>Ct^{-1}|q_{b,0}(X)|))\\
    &\leq (C_3+H|\epsilon|)t^{\zeta}
\end{align*}
where the last inequality follows from Condition~\ref{cond:margin_ZY}. Therefore, the sum is bounded by \[\bigsmile{Ct^{-1}|\epsilon|+H|\epsilon|}  C_2(H|\epsilon|)^{\gamma}+(C_3+H|\epsilon|)t^{\zeta}.\] Taking $t=|\epsilon|^{\frac{1+\gamma}{\zeta+1}}$ gives that this is $O(|\epsilon|^{1+\gamma-\frac{1+\gamma}{\zeta+1}})$, which is $o(|\epsilon|)$ given that $\gamma>\frac{1}{\zeta}$.  
Combining all of the results above gives \[\lim _{\epsilon \rightarrow 0} \frac{1}{\epsilon} \E_{P_{\epsilon}}\left[\left(I\left(q_{b, \epsilon}>0\right)-I\left(q_{b, 0}>0\right)\right) s_{b, \epsilon}\right]=0.\]
Therefore, $\Psi^*$ is pathwise differentiable, and, per \eqref{eqn:decomp_psi}, has the same canonical gradient as the parameter $\Psi_{\pi^*}$, namely $D(\pi^*,P_0)$.
\end{proof}

\begin{proof}[Proof of Theorem~\ref{prop:asymp_linear_psi}]
We would first like to show that $\psi_{OS,n}$ is an asymptotically linear estimator of $\psi_0$. For simplicity of notation, we let $\pi_n^*:=\pi_{\hat{P}_n}^*$ and drop the dependence of $\pi$ in the definition of $\Psi_\pi$ in this proof. Note that $\psi_{OS,n}-\psi_0= (P_n-P_0)D(P_0)+(P_n-P_0)[D(\hat{P}_n)-D(P_0)]+R(\hat{P}_n,P_0)$. Note that the first term $(P_n-P_0)D(P_0)$ is the linear term and $(P_n-P_0)[D(\hat{P}_n)-D(P_0)]=o_{P_0}(n^{-1/2})$ under the Donsker condition and the fact that $\|D(\hat{P}_n)-D(P_0)\|_2\overset{p}{\rightarrow} 0$ (Lemma 19.24 of \cite{van2000asymptotic}). To show that $\psi_{OS,n}$ is asymptotically linear, we only need to argue that the remainder term $R(\hat{P}_n,P_0)$ is $o_{P_0}(n^{-1/2})$. Note that 
\begin{align*}
    P_0 D(\hat{P}_n) &= \E_0\bigbrak{\frac{\I\{A=\pi_n^*(X)\}}{p_n(A|X)}(Y^\dagger-s(A,X))+s(\pi_n^*(X),X)-\Psi(\hat{P}_n)}\\
    &= \E_0\bigbrak{\frac{\I\{A=\pi_n^*(X)\}}{p_n(A|X)}(s_0(A,X)-s(A,X))+s(\pi_n^*(X),X)-\Psi(\hat{P}_n)},
\end{align*}
by the law of total expectation. Therefore, 
\begin{align*}
    R(\hat{P}_n,P_0) &= \Psi(\hat{P}_n)-\Psi(P_0)+P_0 D(\hat{P}_n)\\
    &= \int \bigbrace{\frac{\I\{a=\pi_n^*(x)\}}{p_n(a|x)}(s_0(a,x)-s_n(a,x))+s_n(\pi_n^*(x),x)-s_0(\pi^*(x),x)}dP_0(a,x) \\
    &= \int \left(\frac{\I\{a=\pi_n^*(x)\}}{p_n(a|x)}-1\right)[s_0(\pi_n^*(x),x)-s_n(\pi_n^*(x),x)]dP_0(a,x) + \Psi_{\pi_n^*}(P_0) - \Psi_{\pi^*}(P_0) \\
    &= \iint \left(\frac{\I\{a=\pi_n^*(x)\}}{p_n(a|x)}-1\right)[s_0(\pi_n^*(x),x)-s_n(\pi_n^*(x),x)]p_0(a|x)da\, dP_0(x)\\
    &\quad+ \Psi_{\pi_n^*}(P_0) - \Psi_{\pi^*}(P_0)\\
    &= \int \left(\frac{p_0(\pi_n^*(x)|x)}{p_n(\pi_n^*(x)|x)}-1\right)[s_0(\pi_n^*(x),x)-s_n(\pi_n^*(x),x)]dP_0(x) \\
    &\quad+\Psi_{\pi_n^*}(P_0) - \Psi_{\pi^*}(P_0)\\
    &=: R_{1n}+R_{2n}.
\end{align*}
The first term $R_{1n}$ is $o_{P_0}(n^{-1/2})$ under under Condition~\ref{cond:5} --- see Proposition~\ref{prop:R_1n}. As for the second term $R_{2n}$, Proposition~\ref{prop:R_2n_simple} shows that it is $o_{P_0}(n^{-1/2})$ under the margin condition.
\end{proof}

\begin{proposition}\label{prop:R_1n}
Under Condition~\ref{cond:5}, $R_{1n}=o_{P_0}(n^{-1/2})$. 
\end{proposition}
\begin{proof}
By Jensen's inequality, the fact that $\pi_n^*(x)\in\{0,1\}$ for all $x$, the fact that $(b+c)\le 2\max\{b,c\}$ for $b,c\in\mathbb{R}$, and Cauchy-Schwarz, we have that
\begin{align*}
        |R_{1n}|&=\left|\int \left(\frac{p_0(\pi_n^*(x)|x)}{p_n(\pi_n^*(x)|x)}-1\right)[s_0(\pi_n^*(x),x)-s_n(\pi_n^*(x),x)]dP_0(x)\right|\\
        &\le\int \left|\left(\frac{p_0(\pi_n^*(x)|x)}{p_n(\pi_n^*(x)|x)}-1\right)[s_0(\pi_n^*(x),x)-s_n(\pi_n^*(x),x)]\right|dP_0(x) \\
        &\le\int \sum_{a=0}^1 \left|\left(\frac{p_0(a|x)}{p_n(a|x)}-1\right)[s_0(a,x)-s_n(a,x)]\right|dP_0(x) \\
        &=\sum_{a=0}^1 \int \left|\left(\frac{p_0(a|x)}{p_n(a|x)}-1\right)[s_0(a,x)-s_n(a,x)]\right|dP_0(x) \\
        &\le 2\max_{a\in\{0,1\}}\int \left|\left(\frac{p_0(a|x)}{p_n(a|x)}-1\right)[s_0(a,x)-s_n(a,x)]\right|dP_0(x) \\
        &\leq 2\max_{a\in\{0,1\}} \left\{\left\|\frac{p_0(a \mid X)}{p_{n}(a \mid X)}-1\right\|_{2, P_{0}}\left\|s_n(a, X)-s_{0}(a, X)\right\|_{2, P_{0}}\right\}.
\end{align*}
\end{proof}
The following proposition shows that the second term $R_{2n}$ is $o_{P_0}(n^{-1/2})$ under our margin condition. 
\begin{proposition}\label{prop:R_2n_simple}
Assume Conditions \ref{cond:margin_ZY}, \ref{cond:margin_Y}, and \ref{cond:4} hold. 
Then, for any $\epsilon>0$, $|R_{2n}| = o_{P_0}(n^{-1/2})$.
\end{proposition}
\begin{proof}
We adopt the idea in proof of Theorem 8 of \cite{Alex16}. Let $B_{3,u}'=\{x\in\X:|s_{b,0}(x)|< C_1u|q_{b,0}(x)|\}$ and $A_u=\{x\in\X: C_1u|q_{b,0}(x)|\leq\left| s_{b,0}(x)\right|< C_1(u+1)\left|q_{b,0}(x)\right|\}$. Then for any $t>0$, 
\begin{align*}
|\Psi_{\pi_n^*}\left(P_{0}\right)-\Psi_{\pi^*}\left(P_{0}\right)| &=  \E_{P_0}\left[s_{b, 0}(X)(\pi_n^*(X)-\pi^*(X))\right] \\
&\le \E_{0}\left[\left|s_{b, 0}(X)\right| I\left(\pi^*(X) \neq \pi_n^*(X)\right)\right] \\
&= \sum_{u=0}^\infty \E_0[|s_{b,0}(X)|I(\pi^*(X)\ne \pi_n^*(X))I(A_u)]\\
&\leq \sum_{u=0}^\infty \E_0[|s_{b,0}(X)|I(|q_{b,0}(X)|\leq |q_{b,n}(X)-q_{b,0}(X)|)I(A_u)].
\end{align*}
where the last inequality follows from the fact that for any $x\in\X$, $\pi^*(x) \neq \pi_n^*(x)$ implies that $|q_{b,n}(x)-q_{b,0}(x)|\geq |q_{b,0}(x)|$. From Condition~\ref{cond:margin_ZY} we know that $q_{b,0}(X)\ne 0$ with $P_0$-probability 1, so 
\begin{align*}
    &\sum_{u=0}^\infty \E_0[|s_{b,0}(X)|I(|q_{b,0}(X)|\leq |q_{b,n}(X)-q_{b,0}(X)|)I(A_u)]\\
    &= \sum_{u=0}^\infty \E_0[|s_{b,0}(X)|I(0<|q_{b,0}(X)|\leq |q_{b,n}(X)-q_{b,0}(X)|)I(A_u)].
\end{align*}
For any $x\in A_u$, $|s_{b,0}(x)|\leq C_1(u+1)|q_{b,0}(x)|,$ so for each $u$, 
\begin{align*}
    &\E_{0}\left[\left|s_{b, 0}(X)\right| I\left(0<\left|q_{b, 0}(X)\right| \leq \left|q_{b, n}(X)-q_{b, 0}(X)\right|\right)I(A_u)\right]\\
    &\leq C_1\E_{0}\left[(u+1)\left|q_{b, 0}(X)\right| I\left(0<\left|q_{b, 0}(X)\right| \leq \left|q_{b, n}(X)-q_{b, 0}(X)\right|\right)I(A_u)\right]\\
    &\leq C_1\E_{0}\left[(u+1)\left|q_{b,n}(X)-q_{b, 0}(X)\right| I\left(0<\left|q_{b, 0}(X)\right| \leq \left|q_{b, n}(X)-q_{b, 0}(X)\right|\right)I(A_u)\right]\\
    &\leq C_1\E_{0}\left[(u+1)\max_{x\in\X}\left\|q_{b, n}(x)-q_{b, 0}(x)\right\| I\left(0<\left|q_{b, 0}(X)\right| \leq \max_{x\in\X}\left\|q_{b, n}(x)-q_{b, 0}(x)\right\|\right)I(A_u)\right]\\
    &= C_1(u+1)\left\|q_{b, n}-q_{b, 0}\right\|_{\infty,P_0}\E_{0}\left[I\left(0<\left|q_{b, 0}(X)\right| \leq \max_{x\in\X}\left\|q_{b, n}(x)-q_{b, 0}(x)\right\|\right)I(A_u)\right]\\
    &= C_1(u+1)\left\|q_{b, n}-q_{b, 0}\right\|_{\infty,P_0}P_0(0<\left|q_{b, 0}(X)\right| \leq \left\|q_{b, n}-q_{b, 0}\right\|_{\infty,P_0},A_u).
\end{align*}
For an event $\mathcal{E}\subseteq\mathcal{X}$, let $\P^\infty(\mathcal{E}):=P_0(0<\left|q_{b, 0}(X)\right| \leq \left\|q_{b, n}-q_{b, 0}\right\|_{\infty,P_0},\mathcal{E})$. Then, for any $k\in\mathbb{N}$, 

\begin{align*}
    &\sum_{u=0}^k \E_0[|s_{b,0}(X)|I(0<|q_{b,0}(X)|\leq |q_{b,n}(X)-q_{b,0}(X)|)I(A_u)]
    \\
    &\leq \sum_{u=0}^k C_1(u+1)\left\|q_{b, n}-q_{b, 0}\right\|_{\infty,P_0}\P^\infty(A_u) \\
    &= \sum_{u=0}^k C_1(u+1)\left\|q_{b, n}-q_{b, 0}\right\|_{\infty, P_{0}}[\P^\infty(B_{3,u+1}')-\P^\infty(B_{3,u}')] \\
    &= \sum_{u=0}^k C_1(u+1)\left\|q_{b, n}-q_{b, 0}\right\|_{\infty, P_{0}} \P^\infty(B_{3,u+1}') -\sum_{u=0}^k C_1(u+1)\left\|q_{b, n}-q_{b, 0}\right\|_{\infty, P_{0}}  \P^\infty(B_{3,u}') \\
    &= \sum_{u=1}^{k+1} C_1u\left\|q_{b, n}-q_{b, 0}\right\|_{\infty, P_{0}} \P^\infty(B_{3,u}') -\sum_{u=0}^k C_1(u+1)\left\|q_{b, n}-q_{b, 0}\right\|_{\infty, P_{0}} \P^\infty(B_{3,u}') \\
    &= C_1(k+1)\left\|q_{b, n}-q_{b, 0}\right\|_{\infty, P_{0}} \P^\infty(B_{3,k+1}')-\sum_{u=0}^k C_1\left\|q_{b, n}-q_{b, 0}\right\|_{\infty, P_{0}}\P^\infty(B_{3,u}') \\
    &= \sum_{u=0}^k C_1\left\|q_{b, n}-q_{b, 0}\right\|_{\infty, P_{0}}[\P^\infty(B_{3,k+1}')-\P^\infty(B_{3,u}')] \\
    &\le \sum_{u=0}^k C_1\left\|q_{b, n}-q_{b, 0}\right\|_{\infty, P_{0}}[\P^\infty(\X)-\P^\infty(B_{3,u}')] \\
    &= \sum_{u=0}^k C_1\left\|q_{b, n}-q_{b, 0}\right\|_{\infty, P_{0}}[\P^\infty(B_{3,u}'^c)] \\
    &= \sum_{u=0}^k C_1\left\|q_{b, n}-q_{b, 0}\right\|_{\infty, P_{0}}[\P_0(0<\left|q_{b, 0}(X)\right| \leq \left\|q_{b, n}-q_{b, 0}\right\|_{\infty,P_0}, B_{3,u}'^c)]\\
    &\leq \sum_{u=0}^k C_1\left\|q_{b, n}-q_{b, 0}\right\|_{\infty, P_{0}}^{1+\gamma/2}u^{-\zeta/2}.
\end{align*}
where the last step follows from Holder's inequality. Since $\zeta>2$, let $k\to\infty$ and the infinite sum converges. Therefore,
\begin{align*}
|\Psi_{\pi_n^*}\left(P_{0}\right)-\Psi_{\pi^*}\left(P_{0}\right)| &= \sum_{u=1}^\infty \E_0[|s_{b,0}(X)|I(\pi^*(X)\ne \pi_n^*(X))|A_u]\P(A_u)\\
&=\lim_{k\to\infty} \sum_{u=1}^k \E_0[|s_{b,0}(X)|I(\pi^*(X)\ne \pi_n^*(X))|A_u]\P(A_u)\lesssim \left\|q_{b, n}-q_{b, 0}\right\|_{p, P_{0}}^{1+\gamma/2}.
\end{align*}

Note that under Condition~\ref{cond:4}, we have $\|q_{b,n}-q_{b,0}\|_{\infty,P_0}^{1+\gamma/2}=o_{P_0}(n^{-1/2})$ for any $\gamma>0$, so $|R_{2n}|=o_{P_0}(n^{-1/2})$. 
\end{proof}

\section{Proofs for Section~\ref{sec:general_inference}}\label{sec:technical_proofs}
For notational simplicity, throughout this section and later we denote $\psi_\pi:=\Psi_\pi(P_0)$ for some policy $\pi\in\Pi$. 
\begin{lemma}\label{prop:3.1}
If $\inf_{\pi\in\Pi} \sigma_\pi(P_0)>0$, and $\hat{\sigma}_\pi$ is a consistent estimator of $\sigma_\pi(P_0)$ for each $\pi\in\Pi$, an asymptotically valid uniform $\beta$-level confidence band is given by $\left\{\hat{\omega}_\pi \pm \frac{\hat{\sigma}_\pi t_{\beta}}{n^{1/2}} : \pi\in \Pi\right\}$. 
\end{lemma}
\begin{proof}[Proof of Lemma~\ref{prop:3.1}]
To see that this is the case, note that $t_{\beta}$ is the $1-\beta/2$ quantile of $\sup_{f\in\mathcal{F}} \mathbb{G} f$, and also
\begin{align*}
    P&\cap_{\pi\in\Pi}\left\{\hat{\omega}_\pi - \frac{\hat{\sigma}_\pi t_{\beta}}{n^{1/2}}\le \omega_\pi \le\hat{\omega}_\pi + \frac{\hat{\sigma}_\pi t_{\beta}}{n^{1/2}}\right\} \\
    &= P\cap_{\pi\in\Pi}\left\{-t_\beta\le n^{1/2}\frac{\hat{\omega}_\pi-\omega_\pi}{\hat{\sigma}_\pi} \le t_{\beta}\right\} \\
    &\rightarrow P\cap_{\pi\in\Pi}\left\{-t_\beta\le \mathbb{G}f \le t_{\beta}\right\} \\
    &= P\cap_{\pi\in\Pi}\left[\left\{-t_\beta\le \inf_{f\in\mathcal{F}}\mathbb{G}f\right\}\cap \left\{\sup_{f\in\mathcal{F}}\mathbb{G}f \le t_{\beta}\right\}\right] \\
    &= 1-\beta,
\end{align*}
where the convergence follows from the fact that $n^{1/2}\frac{\hat{\omega}_\pi-\omega_\pi}{\hat{\sigma}_\pi}\rightsquigarrow \GG f$ by Lemma \ref{lem:F_P_Donsk} and Slutsky's Theorem. 
\end{proof}

\begin{proof}[Proof of Lemma~\ref{lem:Pi1b}]
We have that
\[\left\{\Pi^*\subseteq \hat{\Pi}_\beta\right\}=\left\{\omega_{\pi'}<\sup_{\pi\in\Pi}\omega_\pi,\forall \pi'\in \hat{\Pi}_\beta^C\right\}.\]
Therefore, 
\begin{align}
    &\left\{\Pi^*\subseteq \hat{\Pi}_\beta\right\}^C\nonumber\\
    &= \left\{\exists \pi'\in \hat{\Pi}_\beta^C:  \omega_{\pi'}=\sup_{\pi\in\Pi}\omega_\pi\right\}\nonumber\\
    &\subseteq \left\{\exists\pi'\in\hat{\Pi}_\beta^C: \left[\omega_{\pi'}- \hat{\omega}_{\pi'} - \frac{\hat{\sigma}_{\pi'} t_{\beta}}{n^{1/2}} + L_n\right] > \sup_{\pi\in \Pi} \omega_\pi,  \right\} \nonumber \\
    &= \left\{\exists\pi'\in\hat{\Pi}_\beta^C: \left[\omega_{\pi'} - \hat{\omega}_{\pi'} - \frac{\hat{\sigma}_{\pi'} t_{\beta}}{n^{1/2}}\right] > \sup_{\pi\in \Pi} \omega_\pi - L_n \right\}, \label{eq:Pi1_bC}
\end{align}
where the inclusion follows from the 
definition of $\hat{\Pi}_\beta$. Let $\mathcal{A}$ denote the event $\{L_n\le \sup_{\pi\in \Pi}\omega_\pi \}\cap\left[\cap_{\pi\in\Pi}\left\{\omega_\pi \leq\hat{\omega}_\pi + \frac{\hat{\sigma}_\pi t_{\beta}}{n^{1/2}}\right\}\right]$. Hence, \eqref{eq:Pi1_bC} shows that
\begin{align*}
    &\left\{\Pi^*\not\subseteq \hat{\Pi}_\beta\right\}^C \\
    &\subseteq \left[\left\{\exists\pi'\in\hat{\Pi}_\beta^C: \left[\omega_{\pi'} - \hat{\omega}_{\pi'} - \frac{\hat{\sigma}_{\pi'} t_{\beta}}{n^{1/2}}\right] > \sup_{\pi\in \Pi} \omega_\pi - L_n \right\}\cap\mathcal{A}\right]\cup\mathcal{A}^C \\
    &\subseteq \left[\left\{\exists\pi'\in\hat{\Pi}_\beta^C: \omega_{\pi'}- \hat{\omega}_{\pi'} - \frac{\hat{\sigma}_{\pi'} t_{\beta}}{n^{1/2}} > 0\right\}\cap\mathcal{A}\right]\cup\mathcal{A}^C \\
    &= \mathcal{A}^C,
\end{align*}
where the final equality used that the leading event in the union above is equal to the null set since under $\A$, we have $\omega_{\pi'}- \hat{\omega}_{\pi'} - \frac{\hat{\sigma}_{\pi'} t_{\beta}}{n^{1/2}} \leq 0$ for each $\pi\in\Pi$. 
Also, note that by Lemma~\ref{prop:3.1}, $\Pr\left(\cap_{\pi\in\Pi}\left\{\omega_\pi \le\hat{\omega}_\pi + \frac{\hat{\sigma}_\pi t_{\beta}}{n^{1/2}}\right\}\right)\to 1-\beta/2$, and by definition of $L_n$, $\limsup_n\Pr\bigsmile{\left\{L_n< \sup_{\pi\in \Pi} \omega_\pi\right\}}\geq 1-\beta/2$. Hence, by a union bound,
\begin{align*}
    \limsup_n P\left\{\Pi^*\not\subseteq \hat{\Pi}_\beta\right\}&\le \beta.
\end{align*}
\end{proof}

\begin{lemma}\label{lem:PiHat1_b}
For any $\beta>0$, $\liminf_{n\to\infty}\P\left(\omega_{\pi^*}-\inf_{\pi\in\hat{\Pi}_\beta}\omega_\pi\leq \frac{4t_{\beta}}{n^{1/2}}\sup_{\pi\in\Pi}\hat{\sigma}_\pi\right)\geq 1-\beta$. 
\end{lemma}
\begin{proof}[Proof of Lemma~\ref{lem:PiHat1_b}]
Note that by the definition of $\hat{\Pi}_\beta$, we have \[\inf_{\pi\in\hat{\Pi}_\beta}\left[\hat{\omega}_\pi+\frac{\hat{\sigma}_\pi t_{\beta}}{n^{1/2}}\right]\geq \sup_{\pi\in\Pi}\left[\hat{\omega}_\pi-\frac{\hat{\sigma}_\pi t_{\beta}}{n^{1/2}}\right],\] so 
\begin{align*}
    \omega_{\pi^*}-\inf_{\pi\in\hat{\Pi}_\beta}\omega_\pi
    &\leq \omega_{\pi^*}-\sup_{\pi\in\Pi}\left[\hat{\omega}_\pi-\frac{\hat{\sigma}_\pi t_{\beta}}{n^{1/2}}\right]+\inf_{\pi\in\hat{\Pi}_\beta}\left[\hat{\omega}_\pi+\frac{\hat{\sigma}_\pi t_{\beta}}{n^{1/2}}\right]-\inf_{\pi\in\hat{\Pi}_\beta}\omega_\pi.
\end{align*}
Hence,
\begin{align}
    &\left\{\omega_{\pi^*}-\inf_{\pi\in\hat{\Pi}_\beta}\omega_\pi > \frac{4t_{\beta}}{n^{1/2}}\sup_{\pi\in\Pi}\hat{\sigma}_\pi\right\} \nonumber \\
    &\subseteq \left\{\omega_{\pi^*}-\sup_{\pi\in\Pi}\left\{\hat{\omega}_\pi-\frac{\hat{\sigma}_\pi t_{\beta}}{n^{1/2}}\right\}+\inf_{\pi\in\hat{\Pi}_\beta}\left\{\hat{\omega}_\pi+\frac{\hat{\sigma}_\pi t_{\beta}}{n^{1/2}}\right\}-\inf_{\pi\in\hat{\Pi}_\beta}\omega_\pi >\frac{4t_{\beta}}{n^{1/2}}\sup_{\pi\in\Pi}\hat{\sigma}_\pi \right\}  \nonumber \\
    &\subseteq\left\{\omega_{\pi^*}>\sup_{\pi\in\Pi}\left\{\hat{\omega}_\pi - \frac{\hat{\sigma}_\pi t_{\beta}}{n^{1/2}}\right\}+2\sup_{\pi\in\Pi}\frac{\hat{\sigma}_\pi t_{\beta}}{n^{1/2}}\right\} \nonumber \\
    &\qquad\cup \left\{\inf_{\pi\in\hat{\Pi}_\beta}\omega_\pi< \inf_{\pi\in\hat{\Pi}_\beta}\left\{\hat{\omega}_\pi + \frac{\hat{\sigma}_\pi t_{\beta}}{n^{1/2}}\right\}-2\sup_{\pi\in\Pi}\frac{\hat{\sigma}_\pi t_{\beta}}{n^{1/2}}\right\}.  \label{eq:unionBdPiha}
\end{align}
In the remainder of this proof, we will show that the two events on the right-hand side each occur with probability no more than $\beta/2$. The result then follows by a union bound. 
Note that 
\begin{align*}
    \bigcap_{\pi\in\Pi}\left\{\omega_\pi \le\hat{\omega}_\pi + \frac{\hat{\sigma}_\pi t_{\beta}}{n^{1/2}}\right\}&\subseteq \left\{\omega_{\pi^*}\le\sup_{\pi\in\Pi}\left\{\hat{\omega}_\pi + \frac{\hat{\sigma}_\pi t_{\beta}}{n^{1/2}}\right\}\right\}\\
    &\subseteq \left\{\omega_{\pi^*}\le\sup_{\pi\in\Pi}\left\{\hat{\omega}_\pi - \frac{\hat{\sigma}_\pi t_{\beta}}{n^{1/2}}\right\}+2\sup_{\pi\in\Pi}\frac{\hat{\sigma}_\pi t_{\beta}}{n^{1/2}}\right\},
\end{align*}
where the latter inclusion holds because $\sup[f+g]\le \sup f + \sup g$. So
\begin{align*}
    &\liminf_{n\to\infty}\P\left(\omega_{\pi^*}-\sup_{\pi\in\Pi}\left\{\hat{\omega}_\pi-\frac{\hat{\sigma}_\pi t_{\beta}}{n^{1/2}}\right\}\leq 2\sup_{\pi\in\Pi}\frac{\hat{\sigma}_\pi t_{\beta}}{n^{1/2}}\right)\\
    &\geq \liminf_{n\to\infty}\P\left(\bigcap_{\pi\in\Pi}\left\{\omega_\pi \le\hat{\omega}_\pi + \frac{\hat{\sigma}_\pi t_{\beta}}{n^{1/2}}\right\}\right)\geq 1-\frac{\beta}{2},
\end{align*}
where the last step follows from Lemma~\ref{prop:3.1}. Hence, the first event on the right-hand side of \eqref{eq:unionBdPiha} occurs with probability no more than probability $\beta/2$. 
We also have that
\begin{align*}
    &\bigcap_{\pi\in\Pi}\left\{\omega_\pi \ge\hat{\omega}_\pi - \frac{\hat{\sigma}_\pi t_{\beta}}{n^{1/2}}\right\}\subseteq\bigcap_{\pi\in\hat{\Pi}_\beta}\left\{\omega_\pi \ge\hat{\omega}_\pi - \frac{\hat{\sigma}_\pi t_{\beta}}{n^{1/2}}\right\}\\
    &\subseteq \left\{\inf_{\pi\in\hat{\Pi}_\beta}\omega_\pi\ge\inf_{\pi\in\hat{\Pi}_\beta}\left\{\hat{\omega}_\pi - \frac{\hat{\sigma}_\pi t_{\beta}}{n^{1/2}}\right\}\right\}\\
    &\subseteq \left\{\inf_{\pi\in\hat{\Pi}_\beta}\omega_\pi\ge\inf_{\pi\in\hat{\Pi}_\beta}\left\{\hat{\omega}_\pi + \frac{\hat{\sigma}_\pi t_{\beta}}{n^{1/2}}\right\}-2\sup_{\pi\in\hat{\Pi}_\beta}\frac{\hat{\sigma}_\pi t_{\beta}}{n^{1/2}}\right\},
\end{align*}
since $\inf[f-g]\ge \inf f-\sup g$. 
So
\begin{align*}
    &\liminf_{n\to\infty}\P\left(\inf_{\pi\in\hat{\Pi}_\beta}\left\{\hat{\omega}_\pi+\frac{\hat{\sigma}_\pi t_{\beta}}{n^{1/2}}\right\}- 2\sup_{\pi\in\hat{\Pi}_\beta}\frac{\hat{\sigma}_\pi t_{\beta}}{n^{1/2}}\leq \inf_{\pi\in\hat{\Pi}_\beta}\omega_\pi\right)\\
    &\geq \liminf_{n\to\infty}\P\left(\bigcap_{\pi\in\Pi}\left\{\hat{\omega}_\pi - \frac{\hat{\sigma}_\pi t_{\beta}}{n^{1/2}}\le \omega_\pi\right\}\right)\geq 1-\frac{\beta}{2},
\end{align*}
where the last step follows from Lemma~\ref{prop:3.1}. Hence, the second event on the right-hand side of \eqref{eq:unionBdPiha} occurs with probability no more than probability $\beta/2$. 
\end{proof}

In the following lemma, for some subset $\G$ of a space $L^2(Q)$, define the covering number $N(\epsilon,\G,L^2(Q))$ to be the minimal cardinality of an $\epsilon$-cover of $\G$ with respect to the $L^2(Q)$ metric \cite{van1996weak}. Before stating the lemma, we recall that $\F:=\{D_\pi(P_0)/\sigma_\pi(P_0):\pi\in\Pi\}$.
\begin{lemma}[$\F$ is $P_0$-Donsker]\label{lem:F_P_Donsk}
Assume that Conditions~\ref{cond:restrict_policy_class} and \ref{cond:non_vanish_stdev} hold and also that
\begin{enumerate}[label=(\roman*),ref=(\roman*)]
    \item\label{it:1} $\Pi$ satisfies the uniform entropy bound, that is, 
    $\int_0^{\infty} \sup_{Q_X} \sqrt{\log N\left(\varepsilon, \Pi, L^2(Q_X)\right)} d \varepsilon<\infty$, where the supremum is over all finitely supported measures on $\mathcal{X}$;
    \item\label{it:2} there exists $L>0$ such that, for all finitely supported distributions $Q$ of $(X,A,Y)$ with support on $\X\times\{0,1\}\times\Y$, the gradient map $\pi\mapsto D_\pi$ is $L$-Lipschitz, in the sense that, for any $\pi,\pi'\in\Pi$, $\|D_\pi-D_{\pi'}\|_{L^2(Q)}\le L\|\pi-\pi'\|_{L^2(Q_{X})}$, where $Q_X$ is the marginal distribution of $X$ under $Q$;
    \item\label{it:3} $\sup_{\pi\in\Pi}\operatorname{ess}\sup_{x\in\X,a\in\{0,1\},y\in\Y} |D_{\pi}(P_0)(x,a,y)|<\infty$.
\end{enumerate}
Then, the set $\F:=\{D_\pi(P_0)/\sigma_\pi(P_0):\pi\in\Pi\}$ is $P_0$-Donsker.     
\end{lemma}
\begin{proof}[Proof of Lemma~\ref{lem:F_P_Donsk}]
We would like to use Theorem 2.5.2 of \cite{van1996weak}. First, by \ref{it:3} and Condition \ref{cond:non_vanish_stdev}, $$C:=\frac{\sup_{\pi\in\Pi}\operatorname{ess}\sup_{x\in\X,a\in\{0,1\},y\in\Y} |D_{\pi}(P_0)(x,a,y)|}{\inf_{\pi\in\Pi}\sigma_\pi(P_0)}<\infty.$$
Hence, an envelope function for $\F$ is given by the constant function $F(x,a,y)=C$. By \ref{it:2} and properties of covering numbers, for any $Q$ as stated in \ref{it:2} and implied marginal distribution $Q_X$, we have that $N\left(C\varepsilon , \mathcal{F}, L^2(Q)\right)\leq N\left(C\varepsilon/L, \Pi, L^2(Q_X)\right)$. Combining this with \ref{it:1} shows that $\F$ satisfies the uniform entropy bound in the sense that $\int_0^{\infty} \sup_{Q} \sqrt{\log N\left(\varepsilon, \F, L^2(Q)\right)} d \varepsilon<\infty$, where the supremum is over all finitely supported measures on $\X\times\{0,1\}\times \Y$. Hence, $\F$ is $P_0$-Donsker by Theorem 2.5.2 of \cite{van1996weak}.
\end{proof}

\begin{lemma}\label{lem:closed_Pi}
    $\Pi^*$ is a closed subset of $L^2(P_0)$. 
\end{lemma}
\begin{proof}
    Let $(\pi_k)_{k=1}^\infty$ be a $\Pi^*$-valued sequence that converges to some $\pi^*$ in $L^2(P)$. Since $\pi\mapsto \omega_\pi$ is a continuous map from $\{0,1\}^\mathcal{X}$ to $\mathbb{R}$ when the domain is equipped with the $L^2(P)$-topology, $\omega_{\pi_k}\rightarrow\omega_{\pi^*}$. As $\pi_k\in\Pi^*$ for all $k$, $\omega_{\pi_k}=\sup_{\pi\in \Pi} \omega_{\pi}$ for all $k$. Hence, $\omega_{\pi^*}=\sup_{\pi\in \Pi} \omega_{\pi}$. As $\Pi$ is closed, this shows that $\pi^*\in \Pi^*$. Hence, $\Pi^*$ is a closed subset of $L^2(P)$.
\end{proof}

\begin{lemma}\label{lem:Pi_compact}
If $\Pi^*$ is closed in $L^2(P_0)$ and $\Pi^*$ is $P_0$-Donsker, $\Pi^*$ is compact. 
\end{lemma}
\begin{proof}[Proof of Lemma~\ref{lem:Pi_compact}]
Since $\Pi^*$ is $P_0$-Donsker following from $\Pi$ being $P_0$-Donsker, then $\Pi^*$ is totally bounded in $L^2(P_0)$. Also, since $L^2(P_0)$ is complete, $\Pi^*$ being closed implies that $\Pi^*$ is complete. And totally bounded and complete subsets of a metric space are compact, so $\Pi^*$ is compact. 
\end{proof}

\begin{proof}[Proof of Theorem~\ref{thm:PsiPi}]
We have that
\begin{align*}
    &\left\{[\inf_{\pi\in\Pi^*} \psi_\pi,\sup_{\pi\in\Pi^*} \psi_\pi]\not\subseteq {\rm CI}_n\right\} \\
    &= \left\{\inf_{\pi\in \Pi^*} \psi_\pi <\inf_{\pi\in\hat{\Pi}_\beta}\left[\hat{\psi}_\pi - \frac{\hat{\kappa}_\pi z_{\alpha,\beta}}{n^{1/2}}\right]\right\}\cup \left\{\sup_{\pi\in \Pi^*} \psi_\pi >\sup_{\pi\in\hat{\Pi}_\beta}\left[\hat{\psi}_\pi + \frac{\hat{\kappa}_\pi z_{\alpha,\beta}}{n^{1/2}}\right]\right\} \\
    &\subseteq \left\{\inf_{\pi\in \Pi^*} \psi_\pi <\inf_{\pi\in\hat{\Pi}_\beta}\left[\hat{\psi}_\pi - \frac{\hat{\kappa}_\pi z_{\alpha,\beta}}{n^{1/2}}\right],\Pi^*\subseteq \hat{\Pi}_\beta\right\} \\
    &\quad\cup \left\{\sup_{\pi\in \Pi^*} \psi_\pi >\sup_{\pi\in\hat{\Pi}_\beta}\left[\hat{\psi}_\pi + \frac{\hat{\kappa}_\pi z_{\alpha,\beta}}{n^{1/2}}\right],\Pi^*\subseteq \hat{\Pi}_\beta\right\}\cup \left\{\Pi^*\not\subseteq \hat{\Pi}_\beta\right\}.
\end{align*}
Hence, by a union bound and the fact that $\limsup_n (a_n + b_n + c_n)\le \limsup_n a_n + \limsup_n b_n + \limsup_n c_n$, we see that
\begin{align*}
    &\limsup_n P\left\{{\rm CI}_n\not \subseteq [\inf_{\pi\in\Pi^*} \psi_\pi,\sup_{\pi\in\Pi^*} \psi_\pi]\right\} \\
    &\le \limsup_n P\left\{\inf_{\pi\in \Pi^*} \psi_\pi <\inf_{\pi\in\hat{\Pi}_\beta}\left[\hat{\psi}_\pi - \frac{\hat{\kappa}_\pi z_{\alpha,\beta}}{n^{1/2}}\right],\Pi^*\subseteq \hat{\Pi}_\beta\right\} \\
    &\quad+\limsup_n P\left\{\sup_{\pi\in \Pi^*} \psi_\pi >\sup_{\pi\in\hat{\Pi}_\beta}\left[\hat{\psi}_\pi + \frac{\hat{\kappa}_\pi z_{\alpha,\beta}}{n^{1/2}}\right],\Pi^*\subseteq \hat{\Pi}_\beta\right\} \\
    &\quad+ \limsup_n P\left\{\Pi^*\not\subseteq \hat{\Pi}_\beta\right\}.
\end{align*}
The third term is upper bounded by $\beta$ by Lemma~\ref{lem:Pi1b}. In what follows we will show that the first term on the right-hand side is no more than $(\alpha-\beta)/2$. Similar arguments can be used to show that the second term is also no more than $(\alpha-\beta)/2$. By a union bound argument, the sum of three terms is upper bounded by $\alpha$, which completes the proof. 

We begin by noting that, for any $n\in\mathbb{N}$,
\begin{align*}
    &\left\{\inf_{\pi\in \Pi^*} \psi_\pi <\inf_{\pi\in\hat{\Pi}_\beta}\left[\hat{\psi}_\pi - \frac{\hat{\kappa}_\pi z_{\alpha,\beta}}{n^{1/2}}\right],\Pi^*\subseteq \hat{\Pi}_\beta\right\} \\
    &\subseteq \left\{\inf_{\pi\in \Pi^*} \psi_\pi <\inf_{\pi\in\Pi^*}\left[\hat{\psi}_\pi - \frac{\hat{\kappa}_\pi z_{\alpha,\beta}}{n^{1/2}}\right],\Pi^*\subseteq \hat{\Pi}_\beta\right\} \\
    &\subseteq \left\{\inf_{\pi\in \Pi^*} \psi_\pi <\inf_{\pi\in\Pi^*}\left[\hat{\psi}_\pi - \frac{\hat{\kappa}_\pi z_{\alpha,\beta}}{n^{1/2}}\right]\right\}.
\end{align*}
By Lemma~\ref{lem:Pi_compact} and $\pi\mapsto \psi_\pi$ is continuous, there exists a $\pi^\ell$ such that $\psi_{\pi^\ell}=\inf_{\pi\in \Pi^*} \psi_\pi$. Combining this with the above, we see that
\begin{align*}
    \left\{\inf_{\pi\in \Pi^*} \psi_\pi <\inf_{\pi\in\hat{\Pi}_\beta}\left[\hat{\psi}_\pi - \frac{\hat{\kappa}_\pi z_{\alpha,\beta}}{n^{1/2}}\right],\Pi^*\subseteq \hat{\Pi}_\beta\right\} &\subseteq \left\{\psi_{\pi^\ell} <\inf_{\pi\in\Pi^*}\left[\hat{\psi}_\pi - \frac{\hat{\kappa}_\pi z_{\alpha,\beta}}{n^{1/2}}\right]\right\} \\
    &\subseteq \left\{\psi_{\pi^\ell} <\hat{\psi}_{\pi^\ell} - \frac{\hat{\kappa}_{\pi^\ell} z_{\alpha,\beta}}{n^{1/2}}\right\}.
\end{align*}
Then 
\begin{align*}
    P\left(\psi_{\pi^\ell} <\hat{\psi}_{\pi^\ell} - \frac{\hat{\kappa}_{\pi^\ell} z_{\alpha,\beta}}{n^{1/2}}\right)&= P\left(n^{1/2}\frac{\hat{\psi}_{\pi^\ell}-\psi_{\pi^\ell}}{\hat{\kappa}_{\pi^\ell}}>z_{\alpha,\beta}\right).
\end{align*}
By Condition~\ref{cond:non_vanish_stdev}, $\hat{\kappa}_{\pi^\ell}$ is a consistent estimator for $\kappa_{\pi^\ell}(P_0)$. Then with Slutsky's Theorem, $n^{1/2}\frac{\hat{\psi}_{\pi^\ell}-\psi_{\pi^\ell}}{\hat{\kappa}_{\pi^\ell}}\rightsquigarrow \mathbb{G}f_{\pi^\ell}$, so by definition of $z_{\alpha,\beta}$, $P\left(n^{1/2}\frac{\hat{\psi}_{\pi^\ell}-\psi_{\pi^\ell}}{\hat{\kappa}_{\pi^\ell}}>z_{\alpha,\beta}\right)\leq (\alpha-\beta)/2$, and so 
\begin{align*}
     \limsup_{n\to\infty} \P\left\{\inf_{\pi\in \Pi^*} \psi_\pi <\inf_{\pi\in\hat{\Pi}_\beta}\left[\hat{\psi}_\pi - \frac{\hat{\kappa}_\pi z_{\alpha,\beta}}{n^{1/2}}\right],\Pi^*\subseteq \hat{\Pi}_\beta\right\}\le (\alpha-\beta)/2.
\end{align*}
By a symmetric argument, we also have $\limsup_{n\to\infty}P\left\{\sup_{\pi\in \Pi^*} \psi_\pi >\sup_{\pi\in\hat{\Pi}_\beta}\left[\hat{\psi}_\pi + \frac{\hat{\kappa}_\pi z_{\alpha,\beta}}{n^{1/2}}\right],\Pi^*\subseteq \hat{\Pi}_\beta\right\}\leq (\alpha-\beta)/2$. Therefore, an asymptotic $1-\alpha$ confidence interval for $[\psi^l_0,\psi^u_0]$ is
$$\left[\inf_{\pi\in\hat{\Pi}_\beta}\left\{\hat{\psi}_\pi - \frac{\hat{\kappa}_\pi z_{\alpha,\beta}}{n^{1/2}}\right\}, \sup_{\pi\in\hat{\Pi}_\beta}\left\{\hat{\psi}_\pi + \frac{\hat{\kappa}_\pi z_{\alpha,\beta}}{n^{1/2}}\right\}\right].$$ 
\end{proof}

\begin{proof}[Proof of Lemma~\ref{lem:conv_rate_CI}]
To show this lemma, we first define two events $\{\Pi^*\subseteq\hat{\Pi}_\beta\}$ and $\{\omega_{\pi^*}-\inf_{\pi\in\hat{\Pi}_\beta}\omega_\pi\leq \frac{4t_{\beta}}{n^{1/2}}\sup_{\pi\in\Pi}\hat{\sigma}_\pi\}$. These events ensure that all $\Omega$-optimal policies are contained in $\hat{\Pi}_\beta$, and $\hat{\Pi}_\beta$ only contains nearly optimal policies. Lemma~\ref{lem:Pi1b} and \ref{lem:PiHat1_b} ensure that both events happen with probability at least $1-\beta$ asymptotically. The lemma below ensures that our confidence interval shrinks at an $n^{-1/2}$ rate under these events. 
\end{proof}
\begin{lemma}
In the setting of Lemma~\ref{lem:conv_rate_CI}, under the event $\{\Pi^*\subseteq\hat{\Pi}_\beta\}$ and $\{\omega_{\pi^*}-\inf_{\pi\in\hat{\Pi}_\beta}\omega_\pi\leq \frac{4t_{\beta}}{n^{1/2}}\sup_{\pi\in\Pi}\hat{\sigma}_\pi\}$, the width of the confidence interval for $\psi_0$ is $O_p(n^{-1/2})$.
\end{lemma}
\begin{proof}
We first show that $\sup_{\pi\in\hat{\Pi}_\beta}\left[\hat{\psi}_\pi+\frac{\widehat{\kappa}_\pi z_{\alpha,\beta}}{n^{1/2}}\right]=\psi_0+O_p(n^{-1/2})$. We know that $$\sup_{\pi\in\hat{\Pi}_\beta}\left[\hat{\psi}_\pi+\frac{\widehat{\kappa}_\pi z_{\alpha,\beta}}{n^{1/2}}\right]\leq \sup_{\pi\in\hat{\Pi}_\beta}\psi_\pi+\sup_{\pi\in\hat{\Pi}_\beta}\left[\hat{\psi}_\pi-\psi_\pi+\frac{\widehat{\kappa}_\pi z_{\alpha,\beta}}{n^{1/2}}\right].$$
We then show that $\sup_{\pi\in\Pi^*}\psi_\pi-\sup_{\pi\in\hat{\Pi}_\beta}\psi_\pi=O_p(n^{-1/2})$. Consider some $\pi_1\in\Pi^*$ and $\pi_2\in\hat{\Pi}_\beta$. Let $B_{1,0}=\{x\in\X: \pi_1(x)=1, \pi_2(x)=0\}$ and $B_{0,1}=\{x\in\X: \pi_1(x)=0, \pi_2(x)=1\}$. By the definition of $\Pi^*$ we know that $\omega_{\pi_1}\geq \omega_{\pi_2}$, and
\begin{align*}
    \omega_{\pi_1}-\omega_{\pi_2}&=\int \E[Y^*|A=\pi_1(x), x]dP_0(x)-\int \E[Y^*|A=\pi_2(x), x]dP_0(x)\\
    &= \int_{B_{1,0}} q_{b,0}(x)dP_0(x)-\int_{B_{0,1}} q_{b,0}(x)dP_0(x).
\end{align*}
Since $\pi_1\in\Pi^*$ and $\Pi^*$ contains unrestricted optimal policies by assumption, $\omega_{\pi_1}$ is largest among all $\pi\in\Pi$, which implies that for $x\in B_{1,0}$, $q_{b,0}(x)\geq 0$ and for $x\in B_{0,1}$, $q_{b,0}(x)\leq 0$. This gives us 
\begin{align*}
    \omega_{\pi_1}-\omega_{\pi_2}&= \int_{B_{1,0}} |q_{b,0}(x)|dP_0(x)+\int_{B_{0,1}} |q_{b,0}(x)|dP_0(x).
\end{align*}
On the other hand, on the event $\{\Pi^*\subseteq \hat{\Pi}_\beta\}$, we have $\sup_{\pi\in\hat{\Pi}_\beta} \psi_\pi\geq \sup_{\pi\in\Pi^*} \psi_\pi$, and 
\begin{align*}
    \left|\psi_{\pi_2}-\psi_{\pi_1}\right|&=\left|\int \E[Y^\dagger|A=\pi_2(x), x]dP_0(x)-\int \E[Y^\dagger|A=\pi_1(x), x]dP_0(x)\right|\\
    &= \left|\int_{B_{0,1}} s_{b,0}(x)dP_0(x)-\int_{B_{1,0}} s_{b,0}(x)dP_0(x)\right|\\
    &\leq \int_{B_{1,0}} |s_{b,0}(x)|dP_0(x)+\int_{B_{0,1}} |s_{b,0}(x)|dP_0(x)\\
    &\leq C\int_{B_{1,0}} |q_{b,0}(x)|dP_0(x)+C\int_{B_{0,1}} |q_{b,0}(x)|dP_0(x).
\end{align*}
Therefore, $|\psi_{\pi_2}-\psi_{\pi_1}|\leq C(\omega_{\pi_1}-\omega_{\pi_2})$ for some $C<\infty$. Since this holds for any $\pi_1\in\Pi^*$ and $\pi_2\in\hat{\Pi}_\beta$, we have 
$\sup_{\pi\in\hat{\Pi}_\beta}\psi_\pi-\inf_{\pi\in\Pi^*}\psi_{\pi}\leq C(\sup_{\pi\in\Pi^*}\omega_\pi-\inf_{\pi\in\hat{\Pi}_\beta}\omega_\pi)$. 
Under the event $\{\omega_{\pi^*}-\inf_{\pi\in\hat{\Pi}_\beta}\omega_\pi\leq \frac{4t_{\beta}}{n^{1/2}}\sup_{\pi\in\Pi}\hat{\sigma}_\pi\}$, we have that $\sup_{\pi\in\hat{\Pi}_\beta}\psi_\pi-\inf_{\pi\in\Pi^*}\psi_{\pi}\leq C\frac{4t_{\beta}}{n^{1/2}}\sup_{\pi\in\Pi}\hat{\sigma}_\pi$. Under Condition~\ref{cond:non_vanish_stdev}, we know that $\sup_{\pi\in\Pi}\hat{\sigma}_\pi-\sup_{\pi\in\Pi}\sigma_\pi(P_0)=o_p(1)$, so 
\begin{equation}\label{eqn:66}
    \sup_{\pi\in\hat{\Pi}_\beta}\psi_\pi-\inf_{\pi\in\Pi^*}\psi_{\pi}\leq C\frac{4t_{\beta}}{n^{1/2}}\sup_{\pi\in\Pi}\hat{\sigma}_\pi=C\frac{4t_{\beta}}{n^{1/2}}\left(\sup_{\pi\in\Pi}\sigma_\pi(P_0)+o_p(n^{-1/2})\right)=O_p(n^{-1/2}). 
\end{equation}
Under the event $\{\Pi^*\subseteq \hat{\Pi}_\beta\}$, we have $\sup_{\pi\in\hat{\Pi}_\beta}\psi_\pi\geq \sup_{\pi\in\Pi^*}\psi_{\pi}\geq \inf_{\pi\in\Pi^*}\psi_{\pi}$, so we have $\sup_{\pi\in\Pi^*}\psi_\pi-\sup_{\pi\in\hat{\Pi}_\beta}\psi_\pi=O_p(n^{-1/2})$. 
Also, 
\begin{align*}
    \sup_{\pi\in\hat{\Pi}_\beta}\left[\hat{\psi}_\pi-\psi_\pi+\frac{\widehat{\kappa}_\pi z_{\alpha,\beta}}{n^{1/2}}\right]&\leq \sup_{\pi\in\Pi}\left[\hat{\psi}_\pi-\psi_\pi+\frac{\widehat{\kappa}_\pi z_{\alpha,\beta}}{n^{1/2}}\right]\leq \sup_{\pi\in\Pi}\left[\hat{\psi}_\pi-\psi_\pi\right]+\sup_{\pi\in\Pi}\frac{\widehat{\kappa}_\pi z_{\alpha,\beta}}{n^{1/2}}.
\end{align*}
The first term is $O_p(n^{-1/2})$ under Condition~\ref{cond:asymp_linear_est}. As for the second term, under Condition~\ref{cond:non_vanish_stdev}, 
\[\sup_{\pi\in\Pi}\frac{\widehat{\kappa}_\pi z_{\alpha,\beta}}{n^{1/2}}=\sup_{\pi\in\Pi}\frac{\kappa_\pi(P_0) z_{\alpha,\beta}}{n^{1/2}}+o_p(n^{-1/2})=O_p(n^{-1/2}).\]
Therefore, $\sup_{\pi\in\hat{\Pi}_\beta}\left[\hat{\psi}_\pi-\psi_\pi+\frac{\widehat{\kappa}_\pi z_{\alpha,\beta}}{n^{1/2}}\right]=O_p(n^{-1/2})$ and so $\sup_{\pi\in\hat{\Pi}_\beta}\left[\hat{\psi}_\pi+\frac{\widehat{\kappa}_\pi z_{\alpha,\beta}}{n^{1/2}}\right]=\psi_0^u+O_p(n^{-1/2})$ as desired. By symmetry, $\inf_{\pi\in\hat{\Pi}_\beta}\left[\hat{\psi}_\pi-\frac{\widehat{\kappa}_\pi z_{\alpha,\beta}}{n^{1/2}}\right]=\psi_0-O_p(n^{-1/2})$ as well.
\end{proof}

\begin{proof}[Proof of Theorem~\ref{thm:joint_CI}]
To establish this theorem, we show that 
\[\liminf_{n}\P\left(\sup_{\pi\in\Pi^*}\psi_\pi\leq\sup_{\pi\in\hat{\Pi}^\dagger}\left[\hat{\psi}_\pi+\frac{\hat{\kappa}_\pi u_{\alpha}^\dagger}{n^{1/2}}\right]\right)\geq 1-\alpha/2.\]
We can similarly get 
\[\liminf_{n}\P\left(\inf_{\pi\in\Pi^*}\psi_\pi\geq\inf_{\pi\in\hat{\Pi}^\dagger}\left[\hat{\psi}_\pi-\frac{\hat{\kappa}_\pi u_{\alpha}^\dagger}{n^{1/2}}\right]\right)\geq 1-\alpha/2.\]
Combining the two displays gives us the theorem statement.
Note that
\begin{align*}
    &\left\{\sup_{\pi\in\Pi^*}\psi_\pi\leq\sup_{\pi\in\hat{\Pi}^\dagger}\left[\hat{\psi}_\pi+\frac{\hat{\kappa}_\pi u_{\alpha}^\dagger}{n^{1/2}}\right]\right\}\\
    &\supseteq \left\{\sup_{\pi\in\Pi^*}\psi_\pi\leq\sup_{\pi\in\hat{\Pi}^\dagger}\left[\hat{\psi}_\pi+\frac{\hat{\kappa}_\pi u_{\alpha}^\dagger}{n^{1/2}}\right],\Pi^*\subseteq \hat{\Pi}^\dagger\right\}.
\end{align*}
Since $\Pi^*$ is $P_0$-Donsker following from $\Pi$ being $P_0$-Donsker, $\Pi^*$ is totally bounded in $L^2(P_0)$ \cite{Alex16}. Also, since $L^2(P_0)$ is complete, $\Pi^*$ being closed in $L^2(P_0)$ implies that $\Pi^*$ is complete in $L^2(P_0)$. So $\Pi^*$ is compact in $L^2(P_0)$. Combining this with the fact that $\pi\mapsto \psi_\pi$ is continuous implies that there exists a $\pi^u\in \Pi^*$ such that $\psi_{\pi^u}=\sup_{\pi\in \Pi^*} \psi_\pi$. Combining this with the above, we see that
\begin{align}
    &\left\{\sup_{\pi\in\Pi^*}\psi_\pi\leq\sup_{\pi\in\hat{\Pi}^\dagger}\left[\hat{\psi}_\pi+\frac{\hat{\kappa}_\pi u_{\alpha}^\dagger}{n^{1/2}}\right]\right\}\nonumber\\
    &\supseteq\left\{\psi_{\pi^u}\leq\sup_{\pi\in\hat{\Pi}^\dagger}\left[\hat{\psi}_\pi+\frac{\hat{\kappa}_\pi u_{\alpha}^\dagger}{n^{1/2}}\right],\Pi^*\subseteq \hat{\Pi}^\dagger\right\}\nonumber\\
    &\supseteq\left\{\psi_{\pi^u}\leq\hat{\psi}_{\pi^u}+\frac{\hat{\kappa}_{\pi^u} u_{\alpha}^\dagger}{n^{1/2}},\Pi^*\subseteq \hat{\Pi}^\dagger\right\}\nonumber\\
    &= \left\{\psi_{\pi^u}\le\hat{\psi}_{\pi^u}+\frac{\hat{\kappa}_{\pi^u} u_{\alpha}^\dagger}{n^{1/2}},\omega_{\pi'}<\sup_{\pi\in\Pi}\omega_\pi,\forall \pi'\in (\hat{\Pi}^\dagger)^C\right\}.\label{eqn:joint_event}
\end{align}
Note that 
\begin{align}
    &\left\{\omega_{\pi'}<\sup_{\pi\in\Pi}\omega_\pi,\forall \pi'\in (\hat{\Pi}^\dagger)^C\right\}^C=\left\{\exists \pi'\in (\hat{\Pi}^\dagger)^C:  \omega_{\pi'}=\sup_{\pi\in\Pi}\omega_\pi\right\}\nonumber\\
    &\subseteq \left\{\exists\pi'\in(\hat{\Pi}^\dagger)^C: \left[\omega_{\pi'}- \hat{\omega}_{\pi'} - \frac{\hat{\sigma}_{\pi'} t_{\alpha}^\dagger}{n^{1/2}} + \sup_{\pi\in\Pi} \left[\hat{\omega}_\pi-\frac{\hat{\sigma}_\pi s_{\alpha}^\dagger}{n^{1/2}}\right]\right] > \sup_{\pi\in \Pi} \omega_\pi\right\} \nonumber \\
    &= \left\{\exists\pi'\in(\hat{\Pi}^\dagger)^C: \left[\omega_{\pi'} - \hat{\omega}_{\pi'} - \frac{\hat{\sigma}_{\pi'} t_{\alpha}^\dagger}{n^{1/2}}\right] > \sup_{\pi\in \Pi} \omega_\pi - \sup_{\pi\in\Pi} \left[\hat{\omega}_\pi-\frac{\hat{\sigma}_\pi s_{\alpha}^\dagger}{n^{1/2}}\right] \right\},\label{eq:Pin_C} 
\end{align}
where the inclusion follows from the 
definition of $\hat{\Pi}^\dagger$. Let $\mathcal{A}'$ denote the event $$\left\{\sup_{\pi\in\Pi} \left[\hat{\omega}_\pi-\frac{\hat{\sigma}_\pi s_{\alpha}^\dagger}{n^{1/2}}\right]\le \sup_{\pi\in \Pi}\omega_\pi \right\}\bigcap\left[\bigcap_{\pi\in\Pi}\left\{\omega_\pi \leq\hat{\omega}_\pi + \frac{\hat{\sigma}_\pi t_{\alpha}^\dagger}{n^{1/2}}\right\}\right].$$ Hence, \eqref{eq:Pin_C} shows that
\begin{align*}
    &\left\{\exists \pi'\in (\hat{\Pi}^\dagger)^C:  \omega_{\pi'}=\sup_{\pi\in\Pi}\omega_\pi\right\} \\
    &\subseteq \left[\left\{\exists\pi'\in(\hat{\Pi}^\dagger)^C: \omega_{\pi'} - \hat{\omega}_{\pi'} - \frac{\hat{\sigma}_{\pi'} t_\alpha^\dagger}{n^{1/2}} > \sup_{\pi\in \Pi} \omega_\pi - \sup_{\pi\in\Pi} \left[\hat{\omega}_\pi-\frac{\hat{\sigma}_\pi s_{\alpha}^\dagger}{n^{1/2}}\right] \right\}\cap\mathcal{A}'\right]\cup\mathcal{A}'^C \\
    &\subseteq \left[\left\{\exists\pi'\in(\hat{\Pi}^\dagger)^C: \omega_{\pi'}- \hat{\omega}_{\pi'} - \frac{\hat{\sigma}_{\pi'} t_{\alpha}^\dagger}{n^{1/2}} > 0\right\}\cap\mathcal{A}'\right]\cup\mathcal{A}'^C \\
    &= \mathcal{A}'^C.
\end{align*}
For each $\pi\in\Pi$, we define $\widehat{B}_{n,\pi}:=n^{1/2}\frac{\hat{\omega}_\pi-\omega_\pi}{\hat{\sigma}_\pi}$ and $\widetilde{B}_{n,\pi}:=n^{1/2}\frac{\hat{\psi}_\pi-\psi_\pi}{\hat{\kappa}_\pi}$. Then starting from \eqref{eq:Pin_C}, we have
\begin{align*}
    \left\{\omega_{\pi'}<\sup_{\pi\in\Pi}\omega_\pi,\forall \pi'\in (\hat{\Pi}^\dagger)^C\right\} &\supseteq \A' =\left\{\sup_{\pi\in\Pi} \left[\hat{\omega}_\pi-\frac{\hat{\sigma}_\pi s_{\alpha}^\dagger}{n^{1/2}}\right]< \sup_{\pi\in\Pi}\omega_\pi\right\}\bigcap \left[\bigcap_{\pi\in\Pi}\left\{ \omega_\pi\leq\hat{\omega}_\pi+\frac{\hat{\sigma}_\pi t_{\alpha}^\dagger}{n^{1/2}}\right\}\right]\\
    &\supseteq \bigcap_{\pi\in\Pi}\left\{ \hat{\omega}_\pi-\frac{\hat{\sigma}_\pi s_{\alpha}^\dagger}{n^{1/2}}<\omega_\pi<\hat{\omega}_\pi+\frac{\hat{\sigma}_\pi t_{\alpha}^\dagger}{n^{1/2}}\right\}\\
    &= \bigcap_{\pi\in\Pi}\left\{ -t_{\alpha}^\dagger<n^{1/2}\frac{\hat{\omega}_\pi-\omega_\pi}{\hat{\sigma}_\pi}<s_{\alpha}^\dagger\right\}\\
    &= \bigcap_{\pi\in\Pi}\left\{ -t_{\alpha}^\dagger<B_{n,\pi}<s_{\alpha}^\dagger\right\}\\
    &= \{ -t_{\alpha}^\dagger\leq\inf_{\pi\in\Pi}B_{n,\pi}\}\cap \{\sup_{\pi\in\Pi}B_{n,\pi}\leq s_{\alpha}^\dagger\}.
\end{align*}
Using the above to study the event on the right-hand side of \eqref{eqn:joint_event} shows that  
\begin{align}
    &\left\{\psi_{\pi^u}\le\hat{\psi}_{\pi^u}+\frac{\hat{\kappa}_{\pi^u} u_{\alpha}^\dagger}{n^{1/2}},\omega_{\pi'}<\sup_{\pi\in\Pi}\omega_\pi,\forall \pi'\in (\hat{\Pi}^\dagger)^C\right\}\nonumber\\
    &\supseteq \left\{\psi_{\pi^u}<\hat{\psi}_{\pi^u}+\frac{\hat{\kappa}_{\pi^u}u_{\alpha}^\dagger}{n^{1/2}},-t_{\alpha}^\dagger\leq\inf_{\pi\in\Pi}B_{n,\pi},\sup_{\pi\in\Pi}B_{n,\pi}\leq s_{\alpha}^\dagger\right\}\nonumber\\
    &= \left\{\widetilde{B}_{n,\pi^u}>-u_{\alpha}^\dagger,-t_{\alpha}^\dagger\leq \inf_{\pi\in\Pi}B_{n,\pi},\sup_{\pi\in\Pi}B_{n,\pi}\leq s_{\alpha}^\dagger\right\}\label{eqn:inc}.
\end{align}
We know that the choices $(s_{\alpha}^\dagger, t_{\alpha}^\dagger,u_{\alpha}^\dagger)$ satisfy that
\begin{align}
    \inf_{\pi\in\Pi} \P\left\{\inf_{f\in\mathcal{F}}\mathbb{G}f\ge -t_{\alpha}^\dagger,\sup_{f\in\mathcal{F}}\mathbb{G}f\le s_{\alpha}^\dagger,\mathbb{G}\tilde{f}_\pi\ge -u_{\alpha}^\dagger\right\}\ge 1-\alpha/2.\label{eqn:new_joint_cutoff}
\end{align}
Note that by Condition~\ref{cond:asymp_linear_est}, we have 
$\sup_{\pi\in\Pi}\left[n^{1/2}\frac{\hat{\omega}_\pi-\omega_\pi}{\hat{\sigma}_\pi}-\GG_n f_\pi\right]=o_p(1)$ and also $\frac{\hat{\psi}_{\pi^u}-\psi_{\pi^u}}{\hat{\kappa}_{\pi^u}}-\GG_n \tilde{f}_{\pi^u}=o_p(1)$. 
Since $\sup_{f\in\F}\GG_n f\rightsquigarrow \sup_{f\in\F}\GG f$, $\inf_{f\in\F}\GG_n f\rightsquigarrow \inf_{f\in\F}\GG f$, and for each $\pi\in\Pi$, $\hat{\sigma}_\pi$ is a consistent estimator of $\sigma_\pi$, by Slutsky Theorem, we have $\sup_{\pi\in\Pi}B_{n,\pi}\rightsquigarrow \sup_{f\in\F}\GG f$ and $\inf_{\pi\in\Pi}B_{n,\pi}\rightsquigarrow \inf_{f\in\F}\GG f$. Also, since for each $\tilde{f}\in\tilde{\F}$, $\GG_n \tilde{f}\rightsquigarrow \GG \tilde{f}$ and $\hat{\sigma}_\pi$ is a consistent estimator of $\sigma_\pi$, we similarly have $\widetilde{B}_{n,\pi^u}\rightsquigarrow \GG \tilde{f}_{\pi^u}$. 
Combining \eqref{eqn:joint_event}, \eqref{eqn:inc}, and \eqref{eqn:new_joint_cutoff}, we have 
\begin{align*}
    &\liminf_{n}\P\left(\sup_{\pi\in\Pi^*}\psi_\pi<\sup_{\pi\in\hat{\Pi}}\left[\hat{\psi}_\pi+\frac{\hat{\kappa}_\pi u_{\alpha}^\dagger}{n^{1/2}}\right]\right)\\
    &\geq \liminf_{n}\P\left(\widetilde{B}_{n,\pi^u}<u_{\alpha}^\dagger,-s_{\alpha}^\dagger<\inf_{\pi\in\Pi}B_{n,\pi},\sup_{\pi\in\Pi}B_{n,\pi}<t_{\alpha}^\dagger\right)\\
    &\to \P\left(\GG \tilde{f}_{\pi^u}<u_{\alpha}^\dagger,-s_{\alpha}^\dagger<\inf_{f\in\F}\GG f,\sup_{f\in\F}\GG f<t_{\alpha}^\dagger\right)\\
    &\geq \inf_{\pi\in\Pi}\P\left(\GG \tilde{f}_{\pi}<u_{\alpha}^\dagger,-s_{\alpha}^\dagger<\inf_{f\in\F}\GG f,\sup_{f\in\F}\GG f<t_{\alpha}^\dagger\right)=1-\alpha/2.
\end{align*} 
\end{proof}

\section{Multiplier bootstrap}\label{sec:pseudo_mb}

In practice, we use multiplier bootstrap to estimate the quantiles described in Section~\ref{sec:general_inference} and we provide the pseudocodes of the algorithms below. Algorithm~\ref{alg:mb} estimates $t_\beta$ defined just above Lemma~\ref{lem:Pi1b}. Algorithm~\ref{alg:mb_joint} estimates the quantiles described in \eqref{eqn:joint_cutoff}. In this algorithm, we take $s_\alpha^\dagger=t_\alpha^\dagger$ for simplicity and estimate the best $(t_\alpha^\dagger,u_\alpha^\dagger)$ given samples. Both algorithms approximate suprema and infima over sets indexed by $\pi\in \Pi$ by maxima and minima over $\pi$ belonging to a grid approximation of $\Pi$.
\begin{algorithm}[!htb]
\small
\caption{Multiplier bootstrap}
\begin{algorithmic}[1]\label{alg:mb}
\REQUIRE samples $\{(x_i,a_i,y_i)\}_{i=1}^n$, policy set $\Pi$, bootstrap sample size $B$, confidence level $\beta$
\STATE Take a grid estimate $\{\pi_1,\cdots,\pi_K\}$ of $\Pi$
\STATE for each $k\in [K]$, compute normalized one-step estimates $\{o_i^{(\pi_k)}\}_{i=1}^n$ using collected samples $\{(x_i,a_i,y_i)\}_{i=1}^n$
\FOR{$j=1,\cdots,B$}
\STATE get multiplier bootstrap samples $\epsilon_{ij}$ for $i=1,\cdots,n$ and $k=1,\cdots,K$
\STATE compute $n^{-1/2}\sum_{i=1}^n \epsilon_{ij}o_i^{(\pi_k)}$ and denote the result as $f_{\pi_k}^{(j)}$
\ENDFOR
\STATE Apply quantile normalization to $f_{\pi_k}^{(j)}$ across $j$ for each policy $\pi_k$
\STATE compute $\max_{k\in[K]} f_{\pi_k}^{(j)}$ for each $j$ and denote the resulting dataset as $\{t_i\}_{i=1}^B$
\ENSURE $(1-\beta)$-th quantile of $\{t_i\}_{i=1}^B$
\end{algorithmic}
\end{algorithm}

\begin{algorithm}[!htb]
\small
\caption{Multiplier bootstrap for joint probability}
\begin{algorithmic}[1]\label{alg:mb_joint}
\REQUIRE samples $\{(x_i,a_i,y_i,z_i)\}_{i=1}^n$, policy set $\Pi$, bootstrap sample size $B$, confidence level $\alpha$
\STATE Take a grid estimate $\{\pi_1,\cdots,\pi_K\}$ of $\Pi$
\FOR{$k\in [K]$}
\STATE compute normalized one-step estimates $\{o_i^{(\pi_k)}\}_{i=1}^n$ using collected samples $\{(x_i,a_i,y_i)\}_{i=1}^n$
\STATE compute normalized one-step estimates $\{\tilde{o}_i^{(\pi_k)}\}_{i=1}^n$ using collected samples $\{(x_i,a_i,z_i)\}_{i=1}^n$
\ENDFOR

\FOR{$j=1,\cdots,B$}
\STATE get multiplier bootstrap samples $\epsilon_{ik}^{(j)}$ for $i=1,\cdots,n$ and $k=1,\cdots,K$
\STATE compute $n^{-1/2}\sum_{i=1}^n \epsilon_{ik}^{(j)}o_i^{(\pi_k)}$ and denote the result as $f_{\pi_k}^{(j)}$
\STATE compute $n^{-1/2}\sum_{i=1}^n \epsilon_{ik}^{(j)}\tilde{o}_i^{(\pi_k)}$ and denote the result as $\tilde{f}_{\pi_k}^{(j)}$
\ENDFOR
\STATE Apply quantile normalization to $f_{\pi_k}^{(j)}$ and $\tilde{f}_{\pi_k}^{(j)}$ across $j$ for each policy $\pi_k$
\STATE compute $\max_{k\in[K]} f_{\pi_k}^{(j)}$ for each $j$ and denote the results as $\{s_j\}_{j=1}^B$
\STATE compute probability $\P(\max_{k\in[K]}f_{\pi_k}\leq t, \tilde{f}_{\pi_k}\leq u)$ for each $k=1,\cdots,K$ using the $B$ samples
\ENSURE pairs $(t,u)$ such that $\min_{k\in[K]}\P(\max_{k\in[K]}f_{\pi_k}\leq t, \tilde{f}_{\pi_k}\leq u)=1-\alpha$.
\end{algorithmic}
\end{algorithm}
\end{document}